\documentclass[]{article}
\usepackage[a4paper,hmargin={2.5cm,2.5cm},vmargin={1.0cm,2.5cm},includehead]{geometry}

\usepackage{amsmath}
\usepackage{amssymb}
\usepackage{amsfonts}
\usepackage{amsbsy}
\usepackage{amsthm}
\usepackage{mathtools}
\usepackage{bbm}
\usepackage{xfrac}

\usepackage[utf8]{inputenc}
\usepackage[english]{babel}
\usepackage[T1]{fontenc}
\usepackage{stmaryrd}
\usepackage{csquotes}
\usepackage{authblk}

\usepackage{subfigure}
\usepackage{float}
\usepackage{dsfont}
\usepackage{graphicx}
\usepackage{booktabs}
\usepackage{multirow, multicol}
\usepackage{color}
\usepackage{url}
\usepackage{hyperref}
\usepackage{cite}

\usepackage{pdfpages}
\usepackage{xcolor}
\usepackage{enumitem}
\usepackage{cleveref}

\usepackage[linesnumbered,ruled]{algorithm2e}

\newtheorem{theorem}{Theorem}
\newtheorem{corollary}[theorem]{Corollary}
\newtheorem{lemma}[theorem]{Lemma}
\newtheorem{proposition}[theorem]{Proposition}
\newtheorem{remark}[theorem]{Remark}
\newtheorem{definition}[theorem]{Definition}
\newtheorem{claim}[theorem]{Claim}

\makeatletter
\newtheorem*{rep@theorem}{\rep@title}
\newcommand{\newreptheorem}[2]{%
\newenvironment{rep#1}[1]{%
 \def\rep@title{#2 \ref{##1}}%
 \begin{rep@theorem}}%
 {\end{rep@theorem}}}
\makeatother

\newreptheorem{claim}{Claim}
\newreptheorem{definition}{Definition}

\DeclareMathOperator*{\argmax}{arg\,max}

\renewcommand{\P}[1]{\mathbb{P}\left(#1 \right)} 
\newcommand{\E}[1]{\mathbb{E}\left[ #1 \right]} 

\newcommand{\eps}{\varepsilon} 

\newcommand{\aln}[1]{\begin{align}  #1 \end{align}}
\newcommand{\alns}[1]{\begin{align*}  #1 \end{align*}}

\newcommand{\pbr}[1]{\left( #1 \right)} 
\newcommand{\sbr}[1]{\left[ #1\right]}
\newcommand{\cbr}[1]{\left\{ #1\right\}}

\newcommand{\T}{\mathcal{T}}
\newcommand{\A}{\mathcal{A}}
\newcommand{\G}{\mathcal{G}}

\renewcommand{\O}{\mathcal{O}}

\newcommand{\ov}{\mathsf{ov}}

\newcommand{\pihat}{\widehat{\pi}}

\newcommand{\pihatMO}{\widehat{\pi}_{\mathsf{MO}}}

\newcommand{\pistar}{\pi^*}
\newcommand{\mustar}{\mu^*}
\newcommand{\muhat}{\widehat{\mu}}
\newcommand{\muhatMO}{\widehat{\mu}_{\mathsf{MO}}}

\newcommand{\Mbar}{\overline{M}}
\newcommand{\mubar}{\overline{\mu}}
\newcommand{\id}{\mathsf{id}}

\newcommand{\tpi}{\widetilde{\pi}}

\newcommand{\tp}{\widetilde{p}}

\newcommand{\tG}{\widetilde{G}}
\newcommand{\tX}{\widetilde{X}}

\renewcommand{\deg}{\mathsf{deg}}
\newcommand{\core}{\mathsf{core}}
\newcommand{\dom}{\mathsf{dom}}

\newcommand{\ER}{\mathsf{ER}}
\newcommand{\CER}{\mathsf{CER}}
\newcommand{\WCER}{\mathsf{WCG}}
\newcommand{\SCER}{\mathsf{SCG}}

\newcommand{\calT}{\mathcal{T}}
\newcommand{\calN}{\mathcal{N}}

\newcommand{\calH}{\mathcal{H}}

\newcommand{\spl}{\lambda}
\newcommand{\B}{\mathcal{B}}

\definecolor{warningcol}{rgb}{.99,.1,.5}
\definecolor{todocol}{rgb}{.4,.4,.8}
\definecolor{sketchcol}{rgb}{.4,.4,.8}
\definecolor{outlinecol}{rgb}{.8,.4,.3}

\begin{document}

\title{Robust Graph Matching when Nodes are Corrupt}
\author{Taha Ameen and Bruce Hajek 
}
\affil{ \small University of Illinois Urbana-Champaign\\ Electrical and Computer Engineering and Coordinated Science Laboratory \\ Urbana, IL 61801, USA \\ e-mail: \texttt{tahaa3@illinois.edu, b-hajek@illinois.edu}}

\date{\today}

\markboth{Robust Graph Matching when Nodes are Corrupt}{}

\maketitle

\begin{abstract}
    Two models are introduced to investigate graph matching in the presence of corrupt nodes. The weak model, inspired by biological networks, allows one or both networks to have a positive fraction of molecular entities interact \textit{randomly} with their network. For this model, it is shown that no estimator can correctly recover a positive fraction of the corrupt nodes. Necessary conditions for any estimator to correctly identify and match all the uncorrupt nodes are derived, and it is shown that these conditions are also sufficient for the $k$-core estimator.
  
    The strong model, inspired by social networks, permits one or both networks to have a positive fraction of users connect \textit{arbitrarily}. For this model, detection of corrupt nodes is impossible. Even so, we show that if only one of the networks is compromised, then under appropriate conditions, the maximum overlap estimator can correctly match a positive fraction of nodes albeit without explicitly identifying them.
\end{abstract}

\tableofcontents

\section{Introduction} \label{sec: introduction}
Graph matching is the problem of finding the latent correspondence between two edge-correlated networks. It is a ubiquitous problem in machine learning and finds applications in social networks~\cite{narayanan2009deanonymizing, narayanan2008robust}, biological networks~\cite{singh2008global, kazemi2016proper}, natural language processing~\cite{haghighi2005robust} and computer vision~\cite{schellewald2005probabilistic}. Over a decade of progress has led to a sound understanding of the fundamental limits of graph matching in the case of correlated Erd{\H{o}}s-R{\'e}nyi graphs; an overview is provided in~\Cref{subsec: related-work}. 

Graph matching is a noisy version of the graph isomorphism problem, and is motivated by real-world networks often being correlated but non-identical. For instance, the interaction graphs of two social networks (such as Twitter and Flickr) are correlated because users are likely to connect with the same people in both networks. Indeed, it was shown in~\cite{narayanan2009deanonymizing} that the identities of some nodes in the Twitter graph, despite being anonymized, could be recovered simply by matching to the Flickr network. Another example is protein-protein interaction (PPI), where the interactome of an organism is constructed by connecting two interacting proteins with an edge. The interactomes of two closely related species are then correlated through a latent correspondence. Matching these interactomes allows the identification of conserved functional components between the two species~\cite{singh2008global, bandyopadhyay2006systematic}. 

All these networks are more complicated than correlated Erd{\H{o}}s-R{\'e}nyi graphs, and so designing robust algorithms is paramount in practice. In a sense, algorithms for graph matching may themselves be viewed as robust algorithms for graph isomorphism, with the extent of robustness quantified through tolerance to \textit{edge-corruptions}. In the present work, it is argued that robustness towards \textit{node-corruptions} is also an important factor to consider when designing algorithms. For instance, a user's Twitter account may get hacked, causing them to connect and disconnect arbitrarily with other users. Similarly, a protein in a PPI network may interact randomly with other proteins due to a variety of factors. For example, the popular \emph{Yeast two-hybrid} method constructs a PPI network by pairwise examining the interaction between two proteins by fusing them both to a transcription binding domain in the yeast cell~\cite{fionda2019networks}. However, if one of the proteins is itself an unknown transcription factor, then false positive interactions may be recorded. Conversely, if it fails to express, or is toxic to the cell, or requires post-translational modifications that do not take place in yeast cells, then false negatives can occur~\cite{koh2012analyzing}. 

These phenomena are better captured by node corruptions than edge corruptions. Here, the number of corrupted node pairs may even be quadratic in the size of the graph, but there is a spatial clustering of the noise: each corrupted edge has at least one end point in a subset of corrupted nodes.

\paragraph{Contributions} To our knowledge, this is the first work to consider fundamental limits of graph matching with node-corruptions. Two models are studied:
\begin{enumerate}
    \item \textit{Weakly corrupted graphs} ($\WCER$): The adversary selects a random set of nodes in each network and resamples all the edges adjacent to the set without observing the graphs. This models random behavior of unknown proteins in a PPI network.
    \item \textit{Strongly corrupted graphs} ($\SCER$): The adversary selects an arbitrary set of nodes in each network and rewires all the edges adjacent to the set after observing the graphs. This models malicious behavior of hacked users in social networks.
\end{enumerate}

For the $\WCER$ model, we show that no estimator correctly matches any positive fraction of corrupted nodes. Conversely, under appropriate conditions, the $k$-core estimator correctly matches almost all of the uncorrupted nodes and none of the corrupted nodes. Under a further condition that is also necessary, it identifies the corrupted nodes and correctly matches all the uncorrupted nodes. Our simulations suggest that there is a gap between these fundamental limits and the performance of commonly used computationally feasible algorithms.

For the $\SCER$ model, we show that an analogous detection of corrupted nodes is impossible. Even so, when only one of the networks is corrupted, the maximum overlap estimator outputs a matching that correctly matches a positive fraction of the uncorrupted nodes. An explicit lower bound on the fraction of correctly matched nodes as a function of the fraction of corrupted nodes is also derived.

\subsection{Related work} \label{subsec: related-work}
The problem of finding necessary and sufficient conditions for matching correlated random graphs was considered in~\cite{pedarsani2011privacy}. Ever since, a growing line of work has improved these results for exact recovery~\cite{cullina2016improved,cullina2017exact}, almost-exact recovery~\cite{cullina2019kcore, wu2022settling} and partial recovery~\cite{hall2023partial, ganassali2021impossibility, ding2022densesubgraph}. In parallel, other works have investigated computationally feasible algorithms~\cite{barak2019nearly, ding2021degreeprofile, fan2022spectral, mao2021loglog}, culminating in algorithms that run provably well in polynomial time when the graphs are far from isomorphic~\cite{mao2023constant,mao2023chandelier,ding2023polynomial}.

All these works study correlated Erd{\H{o}}s-R{\'e}nyi graphs, for which the fundamental limits of achievability and impossibility are now well understood. Subsequently, an emerging line of work is expanding the scope of the problem. For instance,~\cite{racz2021correlated}~and~\cite{gaudio2022exact} study the graph matching problem in correlated stochastic block models, and~\cite{racz2023matching} studies information theoretic limits of graph matching in inhomogeneous random graphs.

Recently, there is growing interest in studying robust variants of estimation problems in graphs when a positive fraction of nodes are corrupted. For example,~\cite{acharya2022robust} studies the problem of estimating the parameter $p$ of an Erd{\H{o}}s-R{\'e}nyi graph in the presence of such an adversary. More recently,~\cite{liu2022minimax}~and~\cite{hua2023reaching} study the community detection problem when nodes are corrupted. Finally, the model in~\cite{mitzenmacher2018reconciling} allows for a simpler version of node corruptions in graph matching, but studies worst-case performance when a sublinear fraction of nodes are corrupted. All these results provide insight into robustness of algorithms, and facilitate development of algorithms better suited for real-world networks.

\section{Preliminaries} \label{sec: preliminaries}

\paragraph{Notation} 
Let $[n]$ denote the set $ \{1,2,\cdots,n\}$ and let $\binom{[n]}{2}$ denote the set of unordered pairs $ \{\{u,v\}: u,v \in [n] \text{ and } u \neq v\}$. For a graph $G$ on $n$ nodes, assume that its node set $V(G)$ is $[n]$, and so its edge set $E(G)$ is a subset of $\binom{[n]}{2}$. In this work, graphs are undirected and unweighted, so denote $G\{i,j\}= 1$ if $\{i,j\} \in E(G)$ and $0$ otherwise. The graph $G$ is sampled from the Erd{\H{o}}s-R{\'e}nyi distribution, denoted $G \sim \mathsf{ER}(n,p)$, if $G$ has $n$ nodes and each edge in $G$ exists with probability $p$. Let $\pi$ be a permutation on $[n]$ and denote by $G^\pi$ the graph obtained by relabeling nodes in $G$ according to $\pi$, so that
\alns{ 
G\{i,j\} = G^{\pi}\cbr{\pi(i),\pi(j)} \ \forall \cbr{i,j}\in\binom{[n]}{2}.
}
Standard asymptotic notation ($O(\cdot), o(\cdot), \Theta(\cdot), \cdots$) is used throughout, and it is implicit that $n\to\infty$.

In this work, $\mathsf{Bern}(p)$ and $\mathsf{Bin}(n,p)$ denote respectively the Bernoulli and binomial distribution. The hypergeometric distribution is denoted by $\mathsf{HypGeom}(n,k,m)$. A random variable with this distribution counts the number of successes in a sample of $k$ elements drawn without replacement from a population of $n$ individuals, of which $m$ elements are considered successes.

\subsection{Correlated graphs and corruption models}
In all definitions below, $n$ is a positive integer and $p,s$ are in $[0,1]$. Further, $G_1$ and $G_2$ are graphs with $V(G) = [n]$ and $\pistar$ is a permutation on $[n]$. 

\begin{definition}[$\CER$ model] The tuple $(G_1,G_2,\pistar)$ is sampled from the correlated Erd{\H{o}}s-R{\'e}nyi distribution $\CER(n,p,s)$ if two graphs $G_1$ and $G_2'$ are obtained by independently subsampling each edge of a parent graph $G \sim \ER(n,p)$ with probability $s$. Independently, a permutation $\pistar$ is sampled uniformly at random, and $G_2$ is obtained as $G_2 = G_2'^{\pistar}$.
\end{definition}

Marginally, $G_1$ and $G_2$ each follow the $\mathsf{ER}(n,ps)$ distribution. However, the two graphs are edge-wise correlated according to a latent permutation $\pistar$. Next, two models of corruption are presented, motivated respectively by applications in protein-protein interaction and social network de-anonymization.

\begin{definition}[$\WCER$ and $\SCER$ Models]
    Let $(G_1,G_2,\pistar)$ be a sample from $\CER(n,p,s)$. Let $\gamma$ and $\spl$ be in $[0,1]$. Consider an adversary that selects two sets of nodes $\B_1 \subseteq V(G_1)$ and $\B_2 \subseteq V(G_2)$ such that $|\B_1| = \spl \gamma n$ and $|\B_2| = (1\!-\!\spl)\gamma n$. Let $\mathcal{E}_{\B_1}$ (resp. $\mathcal{E}_{\B_2}$) denote all the node pairs adjacent to $\B_1$ (resp. $\B_2$) in $\binom{[n]}{2}$:
    \alns{ 
    \mathcal{E}_{\B_1} &=\cbr{ \cbr{i,j} \in \binom{[n]}{2} : i \in \B_1 \text{ or } j \in \B_1}, \\
    \mathcal{E}_{\B_2} &=\cbr{ \cbr{i,j} \in \binom{[n]}{2} : i \in \B_2 \text{ or } j \in \B_2}.    
    }
    \begin{itemize}
        \item $\WCER$ model: The weak adversary selects $\B_1$ and $\B_2$ uniformly at random and independent of $G_1$ and $G_2$. It then assigns the edge status of each node pair in $\mathcal{E}_{\B_1}$ and $\mathcal{E}_{\B_2}$ independently from the $\mathsf{Bern}(ps)$ distribution. The corrupted graphs are denoted $\tG_1$ and $\tG_2$. The tuple $(\B_1,\B_2,\tG_1,\tG_2,\pistar)$ is said to be distributed according to $\WCER(n,p,s,\gamma,\spl)$.

        \item $\SCER$ model: A strong adversary $\mathsf{A}$ is any rule to select the sets $\B_1$ and $\B_2$, and the edge status of all node pairs in $\mathcal{E}_{\B_1}$ and $\mathcal{E}_{\B_2}$. The corrupted graphs are denoted $\tG_1$ and $\tG_2$, and the tuple $(\B_1,\B_2,\tG_1,\tG_2,\pistar)$ is said to be distributed according to $\SCER(n,p,s,\gamma,\spl,\mathsf{A})$.
    \end{itemize}
\end{definition}
In words, the adversary corrupts a total of $\gamma n$ nodes, of which a fraction $\spl$ are in $G_1$ and the rest are in $G_2$. It then modifies the edge status of each node pair with at least one corrupted end point. Note that the $\WCER$ model defines a joint distribution on $(\B_1,\B_2,\tG_1,\tG_2,\pistar)$. In contrast, for the $\SCER$ model, one must explicitly define an adversary $\mathsf{A}$ to obtain a distribution on $(\B_1,\B_2,\tG_1,\tG_2,\pistar)$.


\subsection{Matchings and estimators} 

\begin{definition}[Matching]
    A \emph{matching} $\mu$ is an injective function with domain $\dom(\mu) \subseteq [n]$ and codomain $[n]$. 
\end{definition}

Note that permutations are matchings with domain equal to $[n]$. An \textit{estimator} $\mathcal{E}$ is a mapping that takes in a pair of corrupted graphs $(\tG_1,\tG_2)$ and outputs a matching $\mu$. In doing so, it attempts to recover the latent permutation $\pistar$ between the uncorrupted graphs $G_1$ and $G_2$. Two estimators that have been studied in the absence of any adversary are the maximum overlap estimator $\widehat{\mathcal{E}}_{\mathsf{MO}}$ and the $k$-core estimator $\widehat{\mathcal{E}}_k$. They are presented next using the following definition.

\begin{definition}[Intersection Graph]
    Let $H_1$ and $H_2$ be two graphs and let $\mu$ be a matching. The intersection graph $H_1 \wedge_{\mu} H_2$ is a graph with node set $\dom(\mu)$, such that for any two nodes $i,j \in \dom(\mu)$, the pair $\{i,j\}$ is an edge in $H_1\wedge_{\mu}H_2$ if and only if $\{i,j\}$ is an edge in $H_1$ and $\{\mu(i),\mu(j)\}$ is an edge in $H_2$.
\end{definition}

\paragraph{Maximum overlap estimator}
For two graphs $H_1$ and $H_2$, the maximum overlap estimator $\widehat{\mathcal{E}}_{\mathsf{MO}}(H_1,H_2)$ outputs a matching $\muhatMO$ that maximizes the number of edges in the corresponding intersection graph:
\alns{ 
\muhatMO \in \argmax_{\mu} |E(H_1 \wedge_{\mu} H_2)|.
}
The maximum overlap matching is the maximum likelihood estimator for exact recovery in the absence of the adversary, and is therefore optimal in that setting.

\paragraph{$k$-core estimator}
The $k$-core of a graph $G$, denoted $\core_k(G)$ is the largest set of vertices $A$ of $G$ such that the induced subgraph on $A$ has minimum degree at least $k$. For any two graphs $H_1$ and $H_2$ and non-negative integer $k$, a matching $\mu$ is said to be a \emph{$k$-core matching} of $H_1$ and $H_2$ if the minimum degree in $H_1 \wedge_{\mu} H_2$ is at least $k$. 

The \emph{$k$-core estimator} $\widehat{\mathcal{E}}_k(H_1,H_2)$ selects a $k$-core matching $\muhat_k$ such that $|\dom(\muhat_k)|$ is at least as large as $|\dom(\mu_k)|$, for any other $k$-core matching $\mu_k$.

\subsection{Recovery objectives}
For a matching $\mu$ and a permutation $\pistar$ on $[n]$, denote by $\ov(\mu,\pistar)$ the overlap between $\mu$ and $\pistar$, i.e. the number of nodes on which $\mu$ and $\pistar$ agree:
\alns{ 
\ov(\mu,\pistar) := \left|\cbr{ i \in \dom(\mu): \mu(i) = \pistar(i)}\right|.
}
Upon observing only the pair of corrupted graphs $(\tG_1,\tG_2)$, the objective is to find a matching $\muhat$ to maximize the overlap between $\muhat$ and the latent permutation $\pistar$.~\Cref{def: recovery} captures this notion. 

\begin{definition}[$\alpha$-recovery] \label{def: recovery}
Let $\alpha \in (0,1]$. An estimator that outputs a matching $\muhat$ is said to achieve
\begin{enumerate}
    \item[(i)] $\alpha$-recovery, if $\P{ \frac{\ov(\muhat,\pistar)}{n} \geq \alpha} = 1 - o(1).$
    \item[(ii)] almost $\alpha$-recovery, if for every $\eps > 0$, $$\P{ \frac{\ov(\muhat,\pistar)}{n} \geq \alpha - \eps } = 1 - o(1).$$
\end{enumerate}
\end{definition}


Graph matching is often a precursor to downstream tasks. Subsequently, an estimator is often useful in practice only if it correctly matches all the nodes in its domain. This concept is made rigorous through the notion of \emph{precision}.

\begin{definition}[Precision]
The \emph{precision} $\rho$ of a matching $\muhat$ is the fraction of the matching that is correct, i.e.
\alns{
\rho(\muhat) := \frac{\ov(\muhat,\pistar)}{|\dom(\muhat)|}.
}
\end{definition}

For a sequence of graph-pairs $(\tG_1,\tG_2)_{n}$ on $n$ vertices, an estimator $\mathcal{E}(\tG_1,\tG_2)$ that outputs a matching $\pihat$ is \emph{precise} if $\P{\rho(\pihat) = 1} = 1-o(1)$. For any $\eps > 0$, it is said to be $\eps$-imprecise if $\P{\rho(\pihat) \leq 1- \eps} = 1-o(1)$.

\section{Main Results} \label{sec: main-results}
    Impossibility and achievability results are presented separately for the $\WCER$ and $\SCER$ models. In all the results, $n$ is a positive integer and $p$, $s$, $\gamma$, $\spl$ are real numbers such that $p \in (0,1)$, $s \in (0,1]$, and $\gamma,\spl \in [0,1]$. For $(\B_1,\B_2,\tG_1,\tG_2,\pistar)$ from the $\WCER$ model, denote by $\B_2'$ the pre-image of $\B_2$ under $\pistar$, i.e. 
$$\B_2' := \{ i \in [n]: \pistar(i) \in\B_2\}.$$

\subsection{Results on the \texorpdfstring{$\WCER$}{WCER} Model}
Our first result is an impossibility result that holds for any estimator.
\begin{theorem} \label{thm: WCER-Impossibility}
Let $(\B_1,\B_2, \tG_1,\tG_2,\pistar)$ be distributed according to $\WCER(n,p,s,\gamma,\spl)$. Let $\mathcal{E}(\tG_1,\tG_2)$ be any estimator that returns a matching $\mu$. Let $\alpha^* = 1- \gamma + \spl(1-\spl)\gamma^2$.
\begin{enumerate}
    \item[(i)] If $\mathcal{E}$ is precise, then $$\P{\dom(\mu) \subseteq (\B_1\cup\B_2')^c} \!=\! 1-o(1).$$
    \item[(ii)] If $\mathcal{E}$ achieves almost $\alpha$-recovery, then $\alpha \leq \alpha^*$.
    \item[(iii)] Let $p = C\log(n)/n$ and $\spl \in \{0,1\}$. If $\mathcal{E}$ is precise and achieves $\alpha^*$-recovery, then $ C \geq 1/(s^2\alpha^*)$.
\end{enumerate} 
\end{theorem}

Part (i) of~\Cref{thm: WCER-Impossibility} states that no precise estimator can correctly recover any of the corrupted nodes with high probability. Part (ii) precludes the possibility of almost $1$-recovery (and therefore also $1$-recovery) when $\gamma > 0$, in stark contrast to known achievability results in the absence of adversary. 

Next, we show that when the average degrees of $G_1$ and $G_2$ are logarithmic in the number of nodes $n$ (i.e. $p = C\log(n)/n$ for some positive constant $C$), the $k$-core estimator performs optimally for an appropriate choice of $k$. Specifically, we prove that there is a threshold $\tau \equiv \tau(s,\gamma, \spl)$ such that if $C > \tau$, then the $k$-core estimator identifies and matches all the uncorrupted nodes. Further, if $C < \tau$, then the estimator matches all but a vanishing fraction of the uncorrupted nodes and none of the corrupted nodes.

\begin{theorem} \label{thm: Achievability-WCER}
Let $C$ be a positive constant and suppose that $p  = C \log(n)/n$. Let $(\B_1,\B_2, \tG_1,\tG_2,\pistar)$ be distributed according to $\WCER(n,p,s,\gamma,\spl)$, and let $\muhat_k$ be the matching output by the $k$-core estimator $\widehat{\mathcal{E}}_k(\tG_1,\tG_2)$ with $k = \sqrt{\log n}$. Let $\alpha^* = 1-\gamma +\spl(1-\spl)\gamma^2$.
\begin{enumerate}
    \item[(i)] 
    $\widehat{\mathcal{E}}_k$ is precise.
    \item[(ii)] If $C > 1/(s^2\alpha^*)$, then $\widehat{\mathcal{E}}_k(\tG_1,\tG_2)$ achieves $\alpha^*$-recovery if $\spl \in \{0,1\}$, and achieves almost $\alpha^*$-recovery if $\spl \in (0,1)$. Further, 
    \aln{ 
    \P{\dom(\muhat_k) = (\B_1\cup\B_2')^c} = 1 - o(1).
    }
    \item[(iii)] If $C < 1/(s^2\alpha^*)$, then $\widehat{\mathcal{E}}_k(\tG_1,\tG_2)$ achieves almost $\alpha^*$-recovery for all $\spl \in [0,1]$. Further,
    \aln{ 
    \P{ \left|(\B_1\cup\B_2')^c \setminus \dom(\muhat_k)\right| = o(n)} = 1 \!-\! o(1). \label{eq: WCER-almost-card}
    }
\end{enumerate}
\end{theorem}

\subsection{Results on the \texorpdfstring{$\SCER$}{SCER} Model}

First, we show that in the $\SCER$ model, no estimator can be precise.
\begin{theorem} \label{thm: SCER-Impossibility}
   Suppose $\gamma > 0$. There exists an adversary $\mathsf{A}'$ for which the output $(\B_1,\B_2,\tG_1,\tG_2,\pistar)$ of $\SCER(n,p,s,\gamma,\spl, \mathsf{A}')$ satisfies the following: If an estimator $\mathcal{E}(\tG_1,\tG_2)$ returns a matching $\mu$ with $|\dom(\mu)| = \Theta(n)$, then $\mathcal{E}$ is $\eps^*$-imprecise, where $\eps^*$ is any real number such that $\eps^* < \frac{\max(\spl,1-\spl)\gamma}{2}$.
\end{theorem}

Thus, no estimator can identify a set $S \subset [n]$ containing a positive fraction of nodes, such that all nodes in $S$ are correctly matched. Despite this, we show that if only one of the networks is compromised, then imprecise recovery of a positive fraction of nodes is possible under appropriate conditions. We state two achievability results below.~\Cref{thm: SCER-Achievability-logdegree} deals with the case when average degrees are logarithmic in $n$, whereas~\Cref{thm: SCER-Achievability-lindegree} deals with the case when $p$ is constant.

\begin{theorem} \label{thm: SCER-Achievability-logdegree}
    Let $\spl \in \{0,1\}$ and $\alpha \in [0,1]$. If
    \aln{ 
    \gamma < \frac{s\pbr{1-\alpha^2}}{4},
    }
    then there exists a constant $C' \equiv C'(\alpha,\gamma)$ such that for all $C > C'$ and $p = C\log(n)/n$, and for all adversaries $\mathsf{A}$ and outputs $(\B_1,\B_2,\tG_1,\tG_2,\pistar)$ of $\SCER(n,p,s,\gamma,\spl,\mathsf{A})$, the maximum overlap estimator achieves $\alpha$-recovery.
\end{theorem}

\begin{theorem} \label{thm: SCER-Achievability-lindegree}
Suppose $p$ is constant, $\spl \in\{0,1\}$ and $\alpha \in [0,1]$. If
\aln{ 
\gamma < 1- \sqrt{1-\frac{s^2p(1-p)(1-\alpha^2)}{2}}, \label{eq: condition-SCER-lin}
}
then for all adversaries $\mathsf{A}$ and outputs $(\B_1,\B_2,\tG_1,\tG_2,\pistar)$ of $\SCER(n,p,s,\gamma,\spl,\mathsf{A})$, the maximum overlap estimator achieves $\alpha$-recovery.
\end{theorem}

\section{Proof Outlines} \label{sec: proofs}
    Proofs for all results in~\Cref{sec: main-results} are outlined, with details deferred to the supplementary material. 

\subsection{The \texorpdfstring{$\WCER$}{WCG} Model}
The performance of the $k$-core estimator against the weak adversary is analyzed. First, the impossibility result is proved using an indistinguishability argument. 

\begin{proof}[Proof of \Cref{thm: WCER-Impossibility}]
(i) It suffices to show that no estimator can correctly match any node $i$ in $\B_1\cup\B_2'$ with high probability. Consider the joint distribution of the collection $\cbr{\tG_1\{i,j\}, \tG_2\{\pistar(i),\pistar(j)\} }_{j\in[n], j \neq i}$. These $2(n-1)$ random variables are each distributed as $\mathsf{Bern}(ps^2)$ and mutually independent, since either $G_1\{i,j\}$ or $G_2\{\pistar(i),\pistar(j)\}$ is resampled because either $i \in \B_1$ or $\pistar(i) \in \B_2$. Further, this joint distribution is the same for all nodes $i$ in the set $\B_1 \cup \B_2'$, and so the nodes within it are statistically indistinguishable. Consequently, no estimator can match any subset $\Mbar$ of nodes in $\B_1\cup\B_2'$ better than random guessing.~\Cref{lem: random-guessing} shows that the random guessing estimator is precise if and only if $\P{\Mbar = \phi} = 1-o(1)$, and the desired result follows.
\\

\noindent (ii) \Cref{lem: random-guessing} implies that any estimator can at best match correctly the set $(\B_1\cup\B_2')^c$ and at most a sublinear number of nodes in $(\B_1 \cup \B_2')$. However, $|\B_1 \cup \B_2'|/n$ converges in probability to $\gamma - \spl(1-\spl)\gamma^2$. This follows from~\Cref{lem: hypgeom}, where it shown that $|\B_1 \cap \B_2'|/n$ converges in probability to $\spl(1-\spl)\gamma$. Since no more than a sublinear number of nodes in $\B_1\cup\B_2'$ are correctly matched, it follows that the fraction of correctly matched nodes, $\alpha$ is strictly upper bounded by $|(\B_1\cup\B_2')^c|/n + \eps$ for every $\eps > 0$. We conclude that $\alpha \leq 1- \gamma + \spl(1-\spl)\gamma^2$, as desired. 
\\

\noindent (iii) Assume $\spl = 1$ so that $\B_2' = \phi$, although a similar proof works for $\spl = 0$. With probability $1-o(1)$:
\alns{ 
\alpha^* n  \stackrel{\text{(a)}}{\leq} \ov(\mu,\pistar) \leq |\dom(\mu)| \stackrel{\text{(b)}}{\leq} |\B_1^c| = \alpha^* n,
}
where (a) is true because $\mathcal{E}$ achieves $\alpha^*$ recovery and (b) uses (i) since $\mathcal{E}$ is a precise estimator. The above string of inequalities are thus equalities. Thus, (i) yields that $\dom(\mu) = \B_1^c$ with probability $1-o(1)$. Since $\mathcal{E}$ achieves $\alpha^*$-recovery, it follows that $\mu$ has correctly matched all the vertices in $\B_1^c$. We show that this is only possible when $C \geq 1/(s^2\alpha^*)$. For a graph $G$ and vertex subset $X \subseteq V(G)$, let $G\vert_{X}$ denote the induced subgraph of $G$ on $X$. Then, with probability $1-o(1)$:
\alns{
H_1 &:= \tG_1\vert_{\dom(\mu)} = \tG_1\vert_{\B_1^c} \stackrel{\text{(c)}}{=} G_1\vert_{\B_1^c}, \\
H_2 &:= \tG_2\vert_{\pistar(\dom(\mu))} = \tG_2\vert_{\pistar(\B_1^c)} \stackrel{\text{(d)}}{=} G_2\vert_{\pistar(\B_1^c)},
}
where (c) is because no node pair in $\tG_1\vert_{\B_1^c}$ is influenced by the adversary, and (d) is because $G_2 = \tG_2$. Thus, $(H_1, H_2, \pistar\vert_{\dom(\mu)}) \sim \CER(\alpha^*n,p,s)$. Recovering $\pistar\vert_{\dom(\mu)}$ is the exact graph recovery problem between $H_1$ and $H_2$, which is impossible whenever $\alpha^*nps^2 < 1$, i.e. whenever $C < 1/(s^2\alpha^*)$~\cite{cullina2017exact,wu2022settling}. 
\end{proof}

Next, a proof sketch for~\Cref{thm: Achievability-WCER} is presented.

\begin{proof}[Proof of~\Cref{thm: Achievability-WCER}] 
(i) The proof is deferred to~\Cref{subsec: deferred-proof}. 
\\

\noindent (ii) The union bound yields for any $\delta \geq 0$,
\alns{ 
\mathbb{P}\big(\widehat{\mathcal{E}}_k \text{ achieves almost }\alpha^*\text{-recovery} \big) \geq 1 \!-\! p_1 \!-\! p_2 \!-\! p_3,
}
where
\aln{
p_1 &= \mathbb{P}\big(\muhat_k \neq \pistar\vert_{\core_k(\tG_1\wedge_{\pistar}\tG_2)} \big), \label{eq: p1}
\\
p_2 &= \mathbb{P}\big(\core_k(\tG_1\wedge_{\pistar}\tG_2) \neq (\B_1\cup\B_2')^c\big), \label{eq: p2}
\\
p_3 &= \mathbb{P}\big(\left|(\B_1\cup\B_2')^c\right| < (1-\delta)\alpha^* n\big. \label{eq: p3}
}
From~\Cref{thm: max-core-is-k-core} and~\Cref{lem: xi-is-small} in the proof of (i), it follows that $p_1 = o(1)$. The bulk of the analysis is to show that $p_2 = o(1)$ whenever $C > 1/(s^2\alpha^*)$. This is shown in~\Cref{lem: k-core-degrees} in the supplementary material. Finally, for any $\delta > 0$, it follows from~\Cref{lem: hypgeom} that $p_3 = o(1)$ for all $\spl \in [0,1]$.

When $\spl \in \{0,1\}$, it holds that $p_3 = 0$ even when $\delta = 0$. This is because either $\B_1 = \phi$ or $\B_2' = \phi$, and therefore $|\B_1\cup\B_2'| = \gamma n = \alpha^* n$ in this setting. However, setting $\delta = 0$ corresponds to achieving $\alpha^*$-recovery. 
\\

\noindent (iii) Let $M^*$ denote $\core_k(\tG_1\wedge_{\pistar}\tG_2)$. The union bound yields for any $\delta \geq 0$,
\alns{ 
\mathbb{P}\big(\widehat{\mathcal{E}}_k \text{ achieves almost }\alpha^*\text{-recovery} \big) \geq 1 \!-\! p_1 \!-\! p_4,
}
where $p_1$ is defined in~\eqref{eq: p1}, and
\aln{
p_4 &= \mathbb{P}\big(|M^*| < (1-\delta)\alpha^* n\big). \label{eq: p4}
}
\Cref{lem: almost-k-core-degrees} shows that $p_4 = o(1)$ for any $\delta > 0$. From part (i) of this theorem, it also follows that $p_1 = o(1)$. Therefore, $\widehat{\mathcal{E}}_k$ achieves almost $\alpha^*$-recovery. It remains to prove~\eqref{eq: WCER-almost-card}, i.e. $\widehat{\mathcal{E}}_k$ recovers all but a vanishing fraction of the uncorrupted nodes. Since $p_1 = o(1)$, it suffices to instead show that for any $\delta' > 0$
\alns{
p_5 := \P{ \left|(\B_1\cup\B_2')^c \setminus M^* \right| > \delta' n} = 1 - o(1).
}
Indeed, denoting $\B \!=\! \B_1\!\cup\!\B_2'$, it follows that for any $\delta' > 0$, and $\eps = \delta/\alpha^*$ that
\alns{ 
 p_5 &\leq \P{M^* \!\not\subseteq \B^c} \!+\! \P{\cbr{|\B^c\!\setminus\! M^*| > \delta n} \cap \cbr{M^*\! \subseteq \B^c} } \\
& \leq o(1) + \P{ |\B^c| - |M^*| > \delta n} \\
& \leq o(1) + \P{|\B^c| > (1+\eps/2)\alpha^* n}  + \P{|M^*| < (1-\eps/2)\alpha^* n}\\
& \stackrel{\text{(a)}}{=} o(1) + o(1) + o(1),
}
where (a) uses both~\Cref{lem: hypgeom} and~\Cref{lem: almost-k-core-degrees}.
\end{proof}

\subsection{The \texorpdfstring{$\SCER$}{SCG} Model}

\textit{Proof of~\Cref{thm: SCER-Impossibility}.}
\IncMargin{1.5em}
\begin{algorithm}[t]
\DontPrintSemicolon
\SetAlgoNlRelativeSize{0}
\caption{Adversary $\mathsf{A}'$ \label{Alg: Imitation_Adversary}}
    \SetKwInOut{Inputs}{Inputs}
    \SetKwInOut{Outputs}{Outputs}
    \Inputs{$G_1$, $G_2$, $\gamma$, $\spl$}
    \Outputs{$\B_1$, $\B_2$, $\tG_1$, $\tG_2$}
    Initialize $\tG_1 = G_1$ and $\tG_2 = G_2$\;
    Select $\B_1 = \cbr{1,2,\cdots, \spl\gamma n}$, and $\B_2$ such that $\B_2' =\cbr{1,2,\cdots,(1-\spl)\gamma n}$\;
    Partition $\B_1$ into $\B_{1,1} := \cbr{1,\cdots, \spl\gamma n/2}$ and $\B_{1,2} := \cbr{\gamma \spl n/2 + 1,\cdots,\spl\gamma n}$\;
    \For{$\{(i,j\}$ such that $i \in \B_{1,1}$ and $j \in [n]$}{
        Set $\tG_1\{i,j\} = G_1\{\spl\gamma n/2 + i,j\}$\;
    }
    \For{$(i,j)$ such that $i \in \B_{1,2}$ and $j \in [n]$}{
        Set $\tG_1\{i,j\} = G_1\{i - \spl \gamma n/2,j\}$\;
    }        
    \textbf{return} $\B_1$, $\B_2$, $\tG_1$, $\tG_2$
\end{algorithm}
Without loss of generality, assume $\spl \geq 1/2$ (else $\spl$ and $1-\spl$ may be interchanged in this proof). Consider the adversary $\mathsf{A}'$ in Algorithm~\ref{Alg: Imitation_Adversary}, which effectively swaps nodes in $\B_{1,1}$ and $\B_{1,2}$. Let $\tpi$ denote the corresponding permutation, i.e. 
\alns{ 
\tpi(i) =
\begin{cases} 
\spl \gamma n/2 + i, & i \in \B_{1,1}\\
i - \spl \gamma n/2, & i \in \B_{1,2}\\
i, & \text{otherwise}
\end{cases}
.
}
It follows that the graphs $G_1^{\tpi}$ and $\tG_1$ are isomorphic and have the same node labeling. Therefore, no estimator will be able to identify any node which is not a fixed point of $\tpi$ with probability greater than $1/2$. 

Let $\mu$ be the matching output by an estimator $\mathcal{E}(\tG_1,\tG_2)$ and let $M$ denote $\dom(\mu)$. Since $|M| = \Theta(n)$, there exists a sequence $(\eps_n)$ such that $|M| = \eps_n n$ and $\liminf_{n\to\infty}\eps_n > 0$. Consider the sets $\mathcal{X} = \cbr{i \in M\cap\B_{1,1} : \mu(i) = \pistar(i)}$ and $\mathcal{Y} = \cbr{i \in M\cap\B_{1,1}^c : \mu(i) = \pistar(i)}$. Let $p_1 := \P{\ov(\mu,\pistar)/|M| > \eps^*}$. Then, for any $\delta > 0$:
\aln{ 
p_1 \! &= \P{\frac{|\mathcal{X}|}{|M|} + \frac{|\mathcal{Y}|}{|M|} > \eps^*}
\\
& \leq \mathbb{P}\left(\frac{|\mathcal{X}|}{|M|} \!>\! \delta \right) \!+\! \mathbb{P}\left( \frac{|\mathcal{Y}|}{|M|} > 1 - (\eps^* + \delta)\right) \label{eq: PLX}
\\
& \stackrel{\text{(a)}}{\leq} 2^{-\delta\eps_n n} + \underbrace{\mathbb{P}\left( \frac{|M\cap\B_{1,1}^c|}{|M|} > 1-\pbr{\eps^*+\delta}\right)}_{(\star)}.
}
Here, (a) is because the pair $(\tG_1, \tG_2)$ does not contain the information needed to correctly match node $i$ in $\B_{1,1}$ with probability more than $1/2$, even if an oracle were to correctly match all other nodes in $\B_{1,1}$. Next,~($\star$) is analyzed. Since it is impossible to determine if a node $i$ is in $\B_{1,1}$, it follows that $|M\cap\B_{1,1}^c| \sim \mathsf{HypGeom}(n,\eps_n n,(1-\spl\gamma/2) n)$. This is because the set $M\cap\B_{1,1}^c$ may be viewed as being constructed by sampling $\eps_n n$ nodes from $[n]$ without replacement, where a sampled node $i$ is labeled a success if and only if $i \in \B_{1,1}^c$. Using standard formulas:
\alns{ 
\E{|M\cap\B_{1,1}^c|} &= \pbr{1 - \frac{\spl \gamma}{2}} \eps_n n, \\
\text{Var}(|M\cap \B_{1,1}^c|) &= \frac{\spl\gamma(1-\spl\gamma)\eps_n(1-\eps_n)}{4} \times \frac{n^2}{n-1}.
}
Choose $\delta$ in~\eqref{eq: PLX} to be sufficiently small so that $\eps^* + \delta < \frac{\spl\gamma}{2}$. Then,
\aln{
(\star) & \leq
\P{ \left|\frac{|M\cap\B_{1,1}^c|}{|M|} - \pbr{1\!-\!\frac{\spl\gamma}{2}} \right|> \frac{\spl\gamma}{2}-\pbr{\eps^* +\delta}\!} \nonumber\\
\!&\leq\! \frac{\spl\gamma(1\!-\!\spl\gamma)(1\!-\!\eps_n)}{4\eps_n \pbr{\spl\gamma/2 - (\eps^* + \delta)} ^2}\times \frac{1}{n-1} \label{eq: PB}
}
where~\eqref{eq: PB} uses Chebyshev's inequality. Finally, since $\liminf_{n\to\infty} \eps_n > 0$, it follows that~\eqref{eq: PB} is $o(1)$. Therefore, $p_1 = o(1)$. This concludes the proof.
\hfill\qedsymbol

Next, the maximum overlap estimator is analyzed when only one network is compromised. Without loss of generality, assume that $\spl = 1$, so that $\tG_2 = G_2$.

Some notation is in order. For any matching $\mu$, let $X(\mu)$ denote the number of edges in $G_1\wedge_{\mu} G_2$. Similarly, let $\tX(\mu)$ denote the number of edges in $\tG_1\wedge_{\mu}\tG_2$. Recall that the maximum overlap matching is defined as $\muhatMO \in \argmax_{\mu} \tX(\mu)$. 

Note that it may be assumed without loss of generality that $\dom(\muhatMO) = [n]$, since $\tX(\cdot)$ satisfies a monotonicity property: extending the domain of any matching to $[n]$ does not decrease the number of edges in $\tG_1\wedge_{\mu}\tG_2$. Assume further, with out loss of generality, that the latent correspondence $\pistar$ is the identity permutation $\id$. Recall that a fixed point of a permutation $\pi$ is an input $i$ such that $\pi(i) = i$. Let $\T^{\alpha}$ (resp. $\T^{\leq\alpha}$) denote the set of all permutations with exactly (resp. at most) $\alpha n$ fixed points. It is shown below that $\muhatMO \notin \T^{\leq\alpha}$. It suffices to prove:
\aln{ 
\P{\tX(\id) > \tX(\pi) \text{ for all } \pi\in\T^{\leq\alpha}} = 1- o(1).
}

The following lemma is a useful ingredient for the proof of both~\Cref{thm: SCER-Achievability-logdegree} and~\Cref{thm: SCER-Achievability-lindegree}.
\begin{lemma} \label{lem: Z-relation}
    For any adversary $\mathsf{A}$, permutation $\pi$, and output $(\B_1,\B_2,\tG_1,\tG_2,\id)$ of $\SCER(n,p,s,\gamma,1,\mathsf{A})$:
    \alns{ 
        \tX(\id) - \tX(\pi) > X(\id) - X(\pi)-Z,
    }
    where
    \aln{ 
        Z := \max_{\substack{S,T\subseteq [n] \\ |S|,|T|\leq\gamma n}} \sum_{\substack{i\in S\cup T\\ j\in[n]}} G_2\{i,j\}. \label{eq: Z-defn}
    }
\end{lemma}
\begin{proof}
For any selection $\B_1$ of $\gamma n$ nodes in $G_1$, recall that $\mathcal{E}_{\B_1} := \cbr{ \cbr{i,j} \in \binom{[n]}{2}: i \in \B_1 \text{ or } j \in \B_2}$. The adversary selects $\B_1$ and sets the edge status of all node pairs in $\mathcal{E}_{\B_1}$ to either $0$ or $1$. For any node pair $e = \{i,j\}$ and graph $H$, let $H(e)$ be a shorthand for $H\{i,j\}$. For any set $\B_1$ and any permutation $\pi$:
\aln{ 
\pbr{\tX(\pi) - \tX(\id)} - \Big(X(\pi)  - X(\id)\Big) &= \sum_{ \substack{ e \in \mathcal{E}_{\B_1}}} \pbr{\tG_1(e) - G_1(e)}\Big(G_2^{\pi}(e) - G_2(e)\Big) \label{eq: diffX}\\
& \stackrel{\text{(a)}}{\leq}  \sum_{ \substack{ e \in \mathcal{E}_{\B_1}}} \mathds{1}\{G_1(e) \!=\! 1\} \cdot\mathds{1}\{G_2(e) \!=\! 1\} \cdot \mathds{1}\{G_2^{\pi}(e) \!=\! 0\} \nonumber \\
& + \   \sum_{ \substack{ e \in \mathcal{E}_{\B_1}}} \mathds{1}\{G_1(e) \!=\! 0\} \cdot\mathds{1}\{G_2(e) \!=\! 0\} \cdot \mathds{1}\{G_2^{\pi}(e) \!=\! 1\}  \nonumber \\
& \leq \sum_{e \in \mathcal{E}_{\B_1}} \Big( \mathds{1}\{ G_2(e) = 1\} + \mathds{1}\{G_2^{\pi}(e) = 1 \}\Big). \label{eq: UB}
}
Here, (a) is because each term of the sum in~\eqref{eq: diffX} is in the set $\{-1,0,1\}$, and equals $1$ if and only if the adversary sets $\tG_1(e) = 1 - G_1(e)$ whenever $(G_1(e),G_2(e),G_2^{\pi}(e))$ is either $(1,1,0)$ or $(0,0,1)$. Note that~\eqref{eq: UB} is maximized when $\B_1$ and $\pi$ are chosen to maximize $|E(G_2) \cap \mathcal{E}_{\B_1}| + |E(G_2^\pi) \cap \mathcal{E}_{\B_1}|$. Therefore,
\alns{ 
\pbr{\tX(\pi) - \tX(\id)} - \Big(X(\pi)  - X(\id)\Big)  &\leq \max_{ \B_1, \pi(\B_1)} \sum_{e \in \mathcal{E}_{\B_1}} \pbr{G_2(e) + G_2^{\pi}(e)}\\
& \stackrel{\text{(b)}}{\leq} \max_{\substack{S,T\subseteq [n] \\ |S|,|T|\leq\gamma n}} \sum_{\substack{i\in S\cup T\\ j\in[n]}} G_2\{i,j\},
}
as desired. Here, (b) follows from the fact that $|\B_1| = |\pi(\B_1)| \leq \gamma n$. This concludes the proof.
\end{proof}

\begin{proof}[Proof of~\Cref{thm: SCER-Achievability-logdegree}]
Let $Z$ be as in~\cref{eq: Z-defn}. Let $\Delta_2$ denote the maximum node degree in the graph $G_2$. Since $|S| + |T| \leq 2\gamma n$, it follows that $Z \leq 2\gamma n \Delta_2$. Let $\mathsf{E}$ denote the error event
\aln{ 
\mathsf{E} =  \bigcup_{\pi \in \T^{\leq \alpha}}\cbr{ \tX(\id) - \tX(\pi) < 0 }. \label{eq: error-event}
}
Applying~\Cref{lem: Z-relation} yields for any $\eps > 0$:
\alns{ 
\P{\mathsf{E}} &\!\leq\! \P{\bigcup_{\pi \in \T^{\leq \alpha}}\!\! \cbr{ X(\id) \!-\! X(\pi) < 2\gamma n \Delta_2 }\!}  \!\leq \sum_{i=1}^{3} p_i,
}
where 
\aln{ 
p_1 & = \P{X(\id) \leq (1-\eps)\binom{n}{2}ps^2}, \label{eq: p1-X} \\
p_2 & = \P{\Delta_2 > (1+\eps)nps}, \\
p_3 & = \mathbb{P}\Bigg(\bigcup_{\pi\in\T^{\leq\alpha}} \Bigg\{ X(\pi) \geq (1-\eps) \binom{n}{2}ps^2 - 2\gamma n \pbr{1+\eps}nps \Bigg\} \Bigg) \label{eq: p3-X}
}
Lemmas~\ref{lem: concentration-of-Xid} and~\ref{lem: concentration-of-degree} show that $p_1 = o(1)$ and $p_2 = o(1)$ for any $\eps > 0$ and sufficiently large $C$. Furthermore,~\Cref{lem: Xpi-decays} shows that $p_3 = o(1)$ for sufficiently small $\eps > 0$ and sufficiently large $C > 0$, whenever $\gamma < s(1-\alpha^2)/4$. This requires a Chernoff argument, using bounds on the moment generating function of $X(\pi)$ obtained by analyzing the orbital decomposition of $\pi$. Altogether, it follows that $\P{\mathsf{E}} = o(1)$, which concludes the proof.
\end{proof}

\begin{proof}[Proof of~\Cref{thm: SCER-Achievability-lindegree}]
Let $Z$ be as defined in~\cref{eq: Z-defn}, and let $\Gamma$ denote the constant $\Gamma = \binom{\gamma n}{2}+\gamma(1-\gamma)n^2$. Since for any choice of $S$ and $T$:
\alns{ 
\left|\cbr{\{i,j\}: i \in S\cup T,  j \in [n] }\right| \leq 2\Gamma,
}
it follows that $Z \leq 2\Gamma$. Let $\mathsf{E}$ denote the error event in~\eqref{eq: error-event}. Consequently, for any $\eps > 0$,
\alns{ 
\P{\mathsf{E}} &\leq \P{\bigcup_{\pi\in\T^{\leq\alpha}}\!\! \cbr{X(\id)-X(\pi) < 2\Gamma}} \leq p_1 + p_4,
}
where
\aln{ 
p_1 &\!=\! \P{X(\id) \leq (1-\eps)\binom{n}{2}ps^2} \\
p_4 &\!=\! \mathbb{P}\Bigg(\bigcup_{\pi\in\T^{\leq\alpha}}\! \Bigg\{ X(\pi) \geq (1\!-\!\eps) \binom{n}{2}ps^2 \!-\! 2 \Gamma \Bigg\} \Bigg) \label{eq: p4-X}
}
\Cref{lem: concentration-of-Xid} shows that $p_1 = o(1)$ for any $\eps > 0$. Furthermore,~\Cref{lem: Xpi-decays-lin}, shows that $p_4 = o(1)$ whenever the condition~\eqref{eq: condition-SCER-lin} is satisfied. This uses similar techniques as the proof of~\Cref{lem: Xpi-decays}, although the resulting sufficient conditions for recovery are quite different. Combining, it follows that $\P{\mathsf{E}} = o(1)$ as desired.
\end{proof}

\section{Discussion} \label{sec: discussion}
    \begin{figure}
    \centering
    \includegraphics[width = 0.65\textwidth]{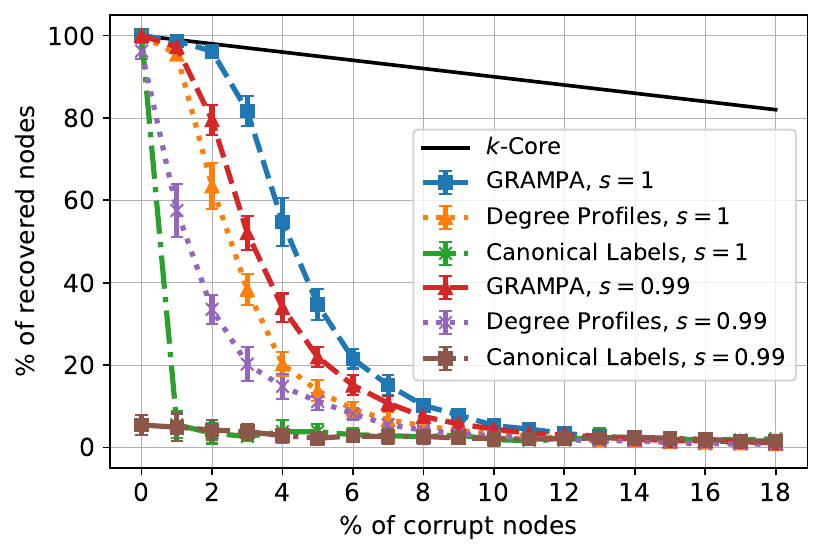}
    \caption{$\WCER$ model, $(n,p,\spl) = (10^3,0.1,1)$}
    \label{fig: comparison}
\end{figure}

\paragraph{Feasible Algorithms}
The maximum overlap and $k$-core estimators are useful to serve as a baseline and to establish theoretical guarantees. However, they do not run in polynomial time, and have limited utility in practice. It is an open question to analyze the performance of \textit{computationally feasible} algorithms when an adversary corrupts nodes.~\Cref{fig: comparison} compares the asymptotic guarantee of the $k$-core estimator against simulation results for the following estimators.
\begin{enumerate}
    \item \textsc{Grampa}~\cite{fan2022spectral} is a spectral algorithm that uses the spectrum of the adjacency matrices to match the two graphs. The code is available in~\cite{fan2020code}.
    \item \textsc{Degree Profiling}~\cite{ding2021degreeprofile} associates with each node a signature, specifically the histogram of the degrees of its neighbors. It then matches nodes based on signature proximity. The code is available in~\cite{ding2020code}.
    \item \textsc{Canonical labeling}~\cite{dai2019canonicallabeling} first matches nodes with outlier degrees, and uses them as seeds to match the remaining nodes.
\end{enumerate}
Clearly, there is large gap between the performance of these algorithms and the asymptotic guarantees of the $k$-core estimator. These algorithms are not robust to the \textit{random} noise in the setting without the adversary, since they require $s\to 1$ for good performance~\cite{fan2022spectral,ding2021degreeprofile,dai2019canonicallabeling}. Perhaps unsurprisingly, they are also not robust to the \textit{spatial} noise induced by the node-based adversary. 

\paragraph{$\SCER$ Model}
\begin{algorithm}[t]
\DontPrintSemicolon
\SetAlgoNlRelativeSize{0}
\caption{Example adversary, $\spl = 1/2$.  \label{Alg: Overwhelm_Adversary}}
    \SetKwInOut{Inputs}{Inputs}
    \SetKwInOut{Outputs}{Outputs}
    \Inputs{$G_1$, $G_2$, $\gamma$, $\spl$}
    \Outputs{$\B_1$, $\B_2$, $\tG_1$, $\tG_2$}
    Initialize $\tG_1 = G_1$ and $\tG_2 = G_2$\;
    Set $\B_1 \!=\! \cbr{1,2,\cdots, \frac{\gamma n}{2}}$, $\B_2 \!=\!\cbr{\frac{\gamma n}{2} \!+\! 1,\cdots,\gamma n}$\;
    Set $\mathcal{G}_1 = \{ i \in [n]: i > \gamma n, i \text{ is odd}\}$ and
    $\mathcal{G}_2 = \{ i \in [n]: i > \gamma n, i \text{ is even}\}$\;
    \For{$k$ in $\{1,2\}$}{
    \For{$\{i,j\}$ such that $i \in \B_{k}$ and $j \in \mathcal{G}_k$}{
        Set $\tG_k\{i,j\} = 1$\;
    }        
    }
    \textbf{return} $\B_1$, $\B_2$, $\tG_1$, $\tG_2$
\end{algorithm}

The $\SCER$ model is significantly more complicated when both networks are compromised. For simplicity, assume that $\spl = 1/2$, and let $p = C\log n/n$ for some positive constant $C$. We show that there is a simple adversarial action that can cause the maximum overlap estimator to recover none of the nodes correctly. Let $(\B_1,\B_2,\tG_1,\tG_2,\id)$ be an output of $\SCER(n,p,s,\gamma,\spl)$ obtained as described in Algorithm~\ref{Alg: Overwhelm_Adversary}. The adversary selects disjoint sets $\B_1$ and $\B_2$ and forces nodes to have high degrees in such a way that the overlap is maximized when every node is wrongly mapped. This is rigorously proved in~\Cref{apx: adv}. One may study variants that pre-process the graphs, but by~\Cref{thm: SCER-Impossibility}, precise recovery of nodes would be impossible even in that setting.



\section{Conclusion} \label{sec: conclusion}
    This work studied two models for graph matching when a positive fraction of nodes interact adversarially with their network. The two models are motivated by practical aspects of network alignment: the $\SCER$ framework models the malicious behavior of hacked users in a social network, whereas the $\WCER$ framework models the random behavior of stochastic interactors in protein-protein interaction networks.

For the $\WCER$ model, our impossibility result states that no positive fraction of the corrupted nodes may be correctly matched. Conversely, under appropriate conditions, the $k$-core estimator correctly matches almost all of the uncorrupted nodes and none of the corrupted nodes. Under a further condition which is necessary, the $k$-core estimator also identifies and recovers all the uncorrupted nodes. In contrast, even the simpler problem of detecting corrupted nodes is impossible to solve in the $\SCER$ model. Even so, the maximum overlap estimator successfully matches a positive fraction of nodes under appropriate conditions.

Looking ahead, feasible algorithms that approach the performance guarantees of the $k$-core estimator in the $\WCER$ model would be useful in practice. One may also study \emph{seeded} robust graph matching, where the correspondence is known for a small subset of nodes. 

\subsubsection*{Acknowledgements}
This work was supported by NSF under Grant CCF 19-00636.

\bibliographystyle{ieeetr}
\bibliography{bibliography.bib}

\newpage
\appendix 

\section{Proofs for the \texorpdfstring{$\WCER$}{WCG} Model} \label{sec: Apx-WCER}

This section presents the proofs pertaining to the $\WCER$ model. First, a standard concentration inequality on binomial random variables is presented. This is used heavily in the remainder of the section, often to bound vertex degrees in various graphs of interest. 
\\

\begin{lemma} \label{lem: concentration-of-binomial}
Let $X \sim \mathsf{Bin}(n,p)$. Then, 
\begin{enumerate}
    \item For any $\delta > 0$, 
    \alns{ 
    \P{X \geq (1+\delta)np} \leq \pbr{\frac{e^\delta}{(1+\delta)^{1+\delta}}}^{np}
    }
    \item For any $\delta \in (0,1)$,
    \alns{ 
    \P{X \leq (1-\delta)np} \leq \pbr{\frac{e^{-\delta}}{(1-\delta)^{1-\delta}}}^{np}
    }
\end{enumerate}
\end{lemma}
\begin{proof}
The proof follows from the Chernoff bound and can be found, for example, in Theorems 4.4 and 4.5 of~\cite{mitzenmacher2017}.
\end{proof}

Next, the lemmas used in the proof of the impossibility result,~\Cref{thm: WCER-Impossibility} are presented.~\Cref{lem: random-guessing} bounds the performance of the estimator that matches vertices in $\B_1\cup\B_2'$ by random guessing, and~\Cref{lem: hypgeom} uses a simple concentration argument to bound the size of $\B_1\cup\B_2'$. Armed with these, we proceed to establish that the $k$-core estimator is precise in~\Cref{subsec: deferred-proof}. It is followed by~\Cref{subsec: def-proofs2}, where supporting lemmas for~\Cref{thm: Achievability-WCER}(ii) and (iii) are respectively presented. 
\\

\begin{lemma} \label{lem: random-guessing}
    Let $n$ be a positive integer, and let $p,s,\gamma,\spl$ be in $[0,1]$. Let $(\B_1,\B_2,\tG_1,\tG_2,\pistar)$ be distributed according to $\WCER(n,p,s,\gamma,\spl)$, and let $\B_2'$ denote the pre-image of $\B_2$ under $\pistar$. Let $\mubar$ be a matching output by an estimator $\overline{\mathcal{E}}$ such that its domain $\dom(\mubar) \subseteq \B_1 \cup \B_2'$ and its codomain is the set $\pistar(\B_1\cup\B_2')$. Conditioned on the domain and codomain, suppose that $\overline{\mathcal{E}}$ matches nodes in its domain by random guessing. Let $\delta > 0$. 
    \begin{enumerate}
        \item[(i)] If $\overline{\mathcal{E}}$ is precise, then $$\P{\dom(\mubar) = \phi} = 1-o(1).$$     
        \item[(ii)] If $\overline{\mathcal{E}}$ is $\delta$-imprecise, then for any $\eps > 0$:
        \aln{ 
            \P{ \ov(\mubar,\pistar) > \eps n } = o(1) \label{eq: gpx}
        }
    \end{enumerate}
\end{lemma}
\vspace{0.5 in}

\begin{proof}
    (i) Note that $\dom(\mubar) \subseteq \B_1\cup\B_2'$ by definition. Since each element in $\dom(\mubar)$ is mapped randomly to an element in $\pistar(\B_1 \cup \B_2')$, and since the mapping is injective, it follows that the probability that all nodes in $\dom(\mubar)$ are correctly matched is given by 
    \alns{ 
    \P{\bigcap_{i \in \dom(\mubar)} \cbr{ \mubar(i) = \pistar(i)}} &= \frac{1}{|\B_1 \cup \B_2'|} \times \frac{1}{|\B_1 \cup \B_2'|-1}\times\cdots\times \frac{1}{|\B_1 \cup \B_2'| - |\dom(\mubar)| + 1} \\
    &= \frac{\pbr{|\B_1 \cup \B_2'| - |\dom(\mubar)|}!}{|\B_1 \cup \B_2'|!}.
    }
    Since $|\B_1\cup\B_2'| \geq \max\pbr{\spl,1-\spl}\gamma n = \Omega(n)$, it follows that the above probability is $o(1)$ if $|\dom(\mubar)| \geq 1$, and equals $1$ if and only if $|\dom(\mubar)| = 0$. Since $\overline{\mathcal{E}}$ achieves precise recovery, the above probability must be $1 - o(1)$, which then implies $\P{\dom(\mubar) = \phi} = 1-o(1)$.

    (ii) For any node $m \in \dom(\mubar)$, let $X_m$ denote the indicator event $\mathds{1}\cbr{\mubar(m) = \pistar(m)}$. Notice that $\ov(\mubar,\pistar) = \sum_{m\in\dom(\mubar)}X_m$, and so it follows that
    \alns{ 
    \E{\ov(\mubar,\pistar)} &= \sum_{m\in \dom(\mubar)}\E{ X_m} = \sum_{m\in\dom(\mubar)} \P{\mubar(m) = \pistar(m)} = \sum_{m\in\dom(\mubar)}\frac{1}{|\dom(\mubar)|} = 1, \\
    \text{Var}(\ov(\mubar,\pistar)) &= \E{Y^2} - 1 = \pbr{\sum_{m_1\in \dom(\mubar)}\sum_{m_2\in\dom(\mubar)} \E{X_{m_1}X_{m_2}}}-1 \stackrel{\text{(a)}}{=} 1,
    }
    where~(a) uses the fact that
    \alns{ 
    \E{X_{m_1} X_{m_2}} = \P{ \cbr{\mubar(m_1) = \pistar(m_1)}\cap\cbr{\mubar(m_2) = \pistar(m_2)}} = 
    \begin{cases}
        \frac{1}{|\dom(\mubar)|} \times \frac{1}{|\dom(\mubar)|-1}, & \text{if } m_1 \neq m_2 \\
        \frac{1}{|\dom(\mubar)|}, &\text{if } m_1 = m_2.
    \end{cases}
    }
    Thus, Chebyshev's inequality yields
    \alns{ 
    \P{ \ov(\mubar,\pistar) \geq \eps n} \leq \P{\left|\ov(\mubar,\pistar) - 1\right| \geq \eps n - 1} \leq \frac{1}{(\eps n - 1)^2} = o(1),
    }
    which concludes the proof.
\end{proof}

\begin{lemma} \label{lem: hypgeom}
    Let $n$ be a positive integer, and let $p,s,\gamma,\spl$ be in $[0,1]$. Let $(\B_1,\B_2,\tG_1,\tG_2,\pistar)$ be distributed according to $\WCER(n,p,s,\gamma,\spl)$, and let $\B_2'$ denote the pre-image of $\B_2$ under $\pistar$. Then, for any $\eps > 0$
    \alns{ 
    \P{ \left| \frac{|\B_1\cap\B_2'|}{n}  - \spl(1-\spl)\gamma^2 \right| > \eps } = o(1).
    }
\end{lemma}
\begin{proof}
    Without loss of generality, assume that $\spl \geq 1/2$, and that $\B_1 = \{1,2,\cdots,\spl\gamma n\}$, since the number of elements in $\B_1 \cap \B_2'$ is independent of the elements in $\B_1$. Since $\pistar$ is independent of the sets $\B_1$ and $\B_2$, selecting $\B_2$ uniformly at random is equivalent to selecting $\B_2'$ uniformly at random. View $\B_2'$ as being constructed by sampling $(1-\spl)\gamma n$ nodes from $[n]$ without replacement. A node $i$ sampled this way is labeled a success if $i \in \B_1$ and a failure otherwise. The number of successes after $(1-\spl)\gamma n$ trials is exactly $|\B_1\cap\B_2'|$, and is described by the hypergeometric distribution $\mathsf{HypGeom}(n,(1-\spl)\gamma n,\spl\gamma n)$. Using standard formulas for the mean and variance of the hypergeometric distribution yields
    \alns{ 
    \E{|\B_1\cap\B_2'|} &= \spl(1-\spl)\gamma^2 n, \\ 
    \text{Var}\pbr{|\B_1\cap\B_2'|} &= \spl(1-\spl)\gamma^2(1-\spl\gamma)(1-(1-\spl)\gamma) \times \frac{n^2}{n-1}.
    }
    Therefore, applying Chebyshev's inequality to $|\B_1\cap\B_2'|/n$ yields that for any constant $\eps > 0$:
    \alns{ 
    \P{  \left| \frac{|\B_1\cap\B_2'|}{n} - \spl(1-\spl)\gamma \right| \geq \eps} \leq \spl(1-\spl)\gamma^2(1-\spl\gamma)(1-(1-\spl)\gamma)\times \frac{1}{\eps^2 (n-1)} = o(1),
    }
    as desired.
\end{proof}

\subsection{Proof of \texorpdfstring{\Cref{thm: Achievability-WCER}(i)}{Theorem}} \label{subsec: deferred-proof}
This subsection analyzes the $k$-core estimator. To show that the $k$-core estimator is precise, it suffices to establish that with high probability, the output of the $k$-core estimator $\muhat_k$ is such that $\dom(\muhat_k)$ is exactly the $k$-core of the true intersection graph $\tG_1\wedge_{\pistar}\tG_2$, and furthermore that the mapping $\muhat_k$ agrees with $\pistar$ on its domain.~\Cref{thm: max-core-is-k-core} and~\Cref{lem: xi-is-small} together establish this. A similar result is proved in~\cite{cullina2019kcore} for Erd{\H{o}}s-R{\'e}nyi graphs when no adversary is present. The techniques introduced there were extended to analyze the $k$-core estimator for graph matching in correlated stochastic block models~\cite{gaudio2022exact} and inhomogeneous random graphs~\cite{racz2023matching}. These techniques are adapted to the $\WCER$ model below. 

\begin{definition}[Weak $k$-core matching]
    Let $H_1$ and $H_2$ be two graphs, and let $\mustar$ and $\mu$ be matchings. We say that $\mu$ is a weak $k$-core matching of $H_1$ and $H_2$ with respect to $\mustar$ if the average degree in $H_1 \wedge_\mu H_2$ of all the nodes $i \in M$ such that $\mu(i) \neq \mustar(i)$ is at least $k$, i.e.
    \alns{ 
    f(\mu;\mustar, H_1,H_2,k) := \sum_{i\in \dom(\mu)\colon \mu(i)\neq\mustar(i)} \deg_{H_1 \wedge_{\mu} H_2 }(i) \geq k \times \left| \cbr{i \in \dom(\mu) \colon \mu(i)\neq\mustar(i)}\right|.
    }
    When the context is clear, we omit the parameters from the notation and simply use $f(\mu)$.
\end{definition}

\begin{definition}[Maximal matching]
    Let $\mustar$ and $\mu$ be matchings. A matching $\mu$ is a maximal matching with respect to $\mustar$ (or simply, a $\mustar$-maximal matching) if for every $i \in \dom(\mustar)$, either $i \in \dom(\mu)$ or $\mustar(i) \in \mathsf{range}(\mu)$.
\end{definition}

\begin{remark}
    Every matching can be uniquely extended to a $\mustar$-maximal matching.
\end{remark}

Denote by $\mathcal{M}(\mustar,d)$ the set of all matchings $\mu$ which are $\mustar$-maximal and such that there are exactly $d$ nodes in $M$ for which the images under $\mu$ and $\mustar$ disagree. Let $\mathcal{M}(\mustar) := \bigcup_{d=0}^n \mathcal{M}(\mustar,d)$.

\begin{lemma} \label{lem: suffices-to-check-maximal-matchings}
    Let $\mustar$ be a matching, and suppose that $\mu$ is a $k$-core matching of two graphs $H_1$ and $H_2$. Then, there exists a matching $\mu' \in \mathcal{M}(\mustar)$ such that $\mu'$ is a weak $k$-core matching. 
\end{lemma}
\begin{proof}
    Since $\mu$ is a $k$-core matching, it is also a weak $k$-core matching with respect to $\mustar$. Let $\mu'$ denote the unique extension of $\mu$ to a $\mustar$-maximal matching. Since the extension only involves adding elements from $\mustar$, it follows that 
    \alns{ 
    k\left|\cbr{i: \mu'(i) \neq \mustar(i)} \right| &= k \left|\cbr{i: \mu(i) \neq \mustar(i) }\right|  \\
    &\leq \sum_{j \in \cbr{i: \mu(i) \neq \mustar(i) }} \deg_{H_1 \wedge_{\mu} H_2}(j)\\
    &\leq \sum_{j \in \cbr{i: \mu'(i) \neq \mustar(i) }} \deg_{H_1\wedge_{\mu'}H_2}(j),
    }
    and so the average degree in $H_1 \wedge_{\mu'} H_2$ of all the nodes $i \in \dom(\mu')$ such that $\mu'(i) \neq \mustar(i)$ it at least $k$. Since $\mu'$ is a $\mustar$-maximal matching by construction, the result follows. 
\end{proof}

Lemma~\ref{lem: suffices-to-check-maximal-matchings} establishes that if no weak $k$-core matchings exist in $\mathcal{M}(\mustar,d)$ for any $d > 0$, then any $k$-core matching must agree with $\mustar$. The main advantage of restricting the search to $\mathcal{M}(\mustar)$ is that there are much fewer $\mustar$-maximal matchings than matchings.

\begin{lemma} \label{lem: size-of-calM}
    $\left|\mathcal{M}(\mustar,d)\right| \leq \frac{n^{2d}}{d!}$.
\end{lemma}
\begin{proof}
    Any $\mustar$-maximal matching can be identified by the set $\{(i,\mu(i)) : \mu(i) \neq \mustar(i) \}$. Since the cardinality of this set is $d$, there are $\binom{n}{d}$ choices of $i$ and at most $\binom{n}{d}$ choices of $\mu(i)$ that preserve the injectivity of $\mu$. Further, there are at most $d!$ ways to match up the $d$ nodes according to $\mu$ that preserve the injectivity of $\mu$. It follows that
    \alns{ 
    \left| \mathcal{M}(\mustar,d)\right| \leq d! \pbr{\binom{n}{d}}^2 \leq d! \pbr{\frac{n^d}{d!}} = \frac{n^{2d}}{d!}.
    }
\end{proof}

\begin{theorem} \label{thm: max-core-is-k-core}
Let $(\B_1,\B_2,\tG_1,\tG_2,\pistar)$ be distributed according to $\WCER(n,p,s,\gamma,\spl)$. Let $\muhat_k$ denote the matching output by the $k$-core estimator $\widehat{\mathcal{E}}_k(\tG_1,\tG_2)$. Then,
\alns{ 
\P{\dom(\muhat_k) = \core_k\pbr{\tG_1 \wedge_{\pistar} \tG_2} \text{ and } \muhat_k = \pistar\mid_{\dom(\muhat_k)}} \geq 2 - \exp\pbr{n^2 \xi},
}
where 
\alns{ 
\xi := \max_{1\leq d\leq n} \max_{(M,\mu)\in\mathcal{M}(d)}\P{f(\mu) \geq kd}^{1/d}.
}
\end{theorem}

\begin{proof}
Let $\mathcal{K}$ denote the set of all $k$-core matchings of $\tG_1$ and $\tG_2$, and let $\mathcal{H}$ denote the event
\alns{ 
\mathcal{H} := \bigcap_{\mu \in \mathcal{K}} \bigcap_{i \in \dom(\mu)} \cbr{\mu(i) = \pistar(i) }
}
First, since $\muhat_k$ itself is a $k$-core matching, the event $\mathcal{H}$ implies that $\muhat_k(i) = \pistar(i)$ for all $i \in\dom(\muhat_k)$. Let $M^* = \core_k(\tG_1\wedge_{\pistar}\tG_2)$. We now show the event $\mathcal{H}$ implies that $\dom(\muhat_k) = M^*$. 

First, note that $\pi^* \vert_{M^*}$ is a $k$-core matching. Therefore, by the maximality property in the definition of the $k$-core estimator, we have that $|\dom(\muhat_k)| \geq |M^*|$. Therefore, to show that $\dom(\muhat_k) = M^*$, it suffices to simply show that $\dom(\muhat_k) \subseteq M^*$ whenever $\mathcal{H}$ occurs. Assume to the contrary that $L := \dom(\muhat_k) \setminus M^*$ is non-empty. Then, on the event $\mathcal{H}$, the subgraph of $\tG_1\wedge_{\pistar}\tG_2$ that is induced on $M^* \cup L$ also has minimum degree $k$, contradicting the maximality of $M^*$. Therefore, $\dom(\muhat_k) = M^*$ if $\mathcal{H}$ occurs. It follows that
\alns{ 
\P{\dom(\muhat_k) = \core_k\pbr{\tG_1 \wedge_{\pistar} \tG_2} \text{ and } \muhat_k = \pistar\mid_{\dom(\muhat_k)}} \geq \P{\mathcal{H}}.
}
To prove the theorem, it suffices to show that $\P{\mathcal{H}^c} \leq \exp\pbr{n^2\xi} - 1$. For any graph $G$, let $d_{\min}(G)$ denote its minimum degree. Indeed, 
\alns{ 
\P{\mathcal{H}^c} &= \P{ \bigcup_{\mu \in \mathcal{K}} \bigcup_{i \in \dom(\mu)} \cbr{ \mu(i) \neq \pistar(i)} } \\
& = \P{\bigcup_{\substack{\mu \text{ such that }\\ \exists i \in \dom(\mu): \mu(i) \neq \pistar(i) }} d_{\min}(\tG_1\wedge_{\mu}\tG_2) \geq k} \\
& \stackrel{\text{(a)}}{\leq} \sum_{d=1}^n \P{\bigcup_{\substack{\mu \text{ such that }\\ \left|\{ i \in \dom(\mu): \mu(i) \neq \pistar(i) \}\right| = d}} f(\mu) \geq kd } \\
& \stackrel{\text{(b)}}{\leq} \sum_{d=1}^n \P{\bigcup_{\mu \in \mathcal{M}(\pistar,d)} f(\mu) \geq kd} \\
& \stackrel{\text{(c)}}{\leq} \sum_{d=1}^n |\mathcal{M}(d)|\pbr{ \max_{1\leq d\leq n}\max_{\mu\in \mathcal{M}(\pistar,d)}\P{f(\mu) \geq kd} }\\
& \stackrel{\text{(d)}}{\leq} \sum_{d=1}^n \frac{(n^2\xi)^d}{d!} \\
& \leq \exp\pbr{n^2\xi} - 1.
}
where (a) partitions the set of all matchings based on the number of disagreements of the matching with $\pistar$ and uses a union bound, (b) follows from~\Cref{lem: suffices-to-check-maximal-matchings}, (c) follows from a union bound, and (d) is from the definition of $\xi$ and an application of~\Cref{lem: size-of-calM}.
\end{proof}

Next, we show that $\xi$ decays rather quickly under appropriate conditions. The proof below follows~\cite{gaudio2022exact} and~\cite{racz2023matching}, although their focus is the stochastic block model and the inhomogeneous random graph model respectively. Below, we adapt the argument to the $\WCER$ model.

\begin{lemma} \label{lem: xi-is-small}
Let $(\B_1,\B_2,\tG_1,\tG_2,\pistar)$ be distributed according to $\WCER(n,p,s,\gamma,\spl)$. Let $\mathcal{M}(\pistar,d)$ be the set of all $\pistar$-maximal matchings $\mu$ such that $|\{i\in \dom(\mu)\colon \mu(i)\neq \pistar(i)\}| = d$. 
\alns{ 
\xi = \max_{1\leq d\leq n} \max_{\mu\in\mathcal{M}(\pistar,d)}\P{f(\mu) \geq kd}^{1/d}.
}
Suppose that $p = \frac{C\log n}{n}$ for some $C > 0$. If $k \geq 13$, then $\xi = o(n^{-2})$.
\end{lemma}
\begin{proof}
For any matching $\mu \in \mathcal{M}(\pistar,d)$, define the following sets:
\alns{ 
\A(\mu)   &:= \cbr{(i,j)\in \dom(\mu)\times \dom(\mu) \colon \mu(i) \neq\pistar(i)}, \\ 
\calT(\mu) &:= \cbr{(i,j) \in \A(\mu) \colon \mu(i) = \pistar(j) \text{ and } \mu(j) = \pistar(i)}, \\
\calN(\mu) &:= \A(\mu) \setminus \calT(\mu).
}
For a graph $G$, let $G(i,j)$ denote the $(i,j)$-th entry of its adjacency matrix. First, note that $\A(\mu) \leq d|\dom(\mu)| \leq dn$, and that $|\calT(\mu)| \leq d$. This is because $\mu$ makes $d$ errors by definition. It follows that
\alns{ 
f(\mu) &= \sum_{\substack{i\in \dom(\mu) \\ \mu(i)\neq\pistar(i)}} \deg_{\tG_1\wedge_{\mu}\tG_2}(i) \\
&= \sum_{(i,j)\in \A(\mu)} \tG_1(i,j) \tG_2(\mu(i),\mu(j)) \\
&= 2\sum_{\substack{(i,j)\in\calT(\mu)\\ i< j}} \tG_1(i,j) \tG_2(\mu(i),\mu(j)) + \sum_{(i,j)\in \calN(\mu)} \tG_1(i,j) \tG_2(\mu(i),\mu(j)) \\
&=: 2X_{\calT} + X_{\calN}.
}
It is easy to see that $X_{\calT}$ and $X_{\calN}$ are independent, since they involve disjoint node pairs. Furthermore, for the same reason, the individual terms in $X_{\calT}$ are also independent. More importantly, 
\alns{ 
\tG_1(i,j) \tG_2(\mu(i),\mu(j)) \sim 
\begin{cases}
    \mathsf{Bern}(p^2 s^2), & i \in \B_1\cup\B_2'  \text{ or } j \in \B_1\cup\B_2' \\
    \mathsf{Bern}(p s^2), & \text{otherwise}
\end{cases}
.
}
Therefore, it follows that
\aln{ 
X_{\calT} \preceq \mathsf{Bin}(|\calT(\mu)|, ps^2) \preceq \text{Bin}(d,ps^2), \label{eq: XT-domination}
}
where $\preceq$ denotes stochastic domination of the RHS. Next, the $X_{\calN}$ is analyzed, which is slightly more complicated because the summands may be correlated. To circumvent this, partition $\calN(\mu)$ into $\calN_1(\mu)$, $\calN_{2}(\mu)$, and $\calN_{3}(\mu)$ and define
\alns{ 
X_{\calN_{j}} := \sum_{\substack{(i,j)\in\calN_{j}(\mu) \\ i < j}} \tG_1(i,j) \tG_2(\mu(i),\mu(j)), \quad j \in \{1,2,3\}.
} 
It follows that $X_{\calN} \leq 2(X_{\calN_1}+X_{\calN_2}+X_{\calN_3})$. Here, the factor of $2$ accounts for the restriction that $i < j$. Furthermore, it is possible to partition $\calN$ in such a way that $X_{\calN_1}$, $X_{\calN_2}$ and $X_{\calN_3}$ are mutually independent. To see this, consider two node pairs $(i,j)$ and $(a,b)$ in $\binom{M}{2}$. The random variables $\tG_1(i,j) \tG_2(\mu(i),\mu(j))$ and $\tG_1(a,b) \tG_2(\mu(a),\mu(b))$ are dependent if and only if one of the following two conditions hold:
\aln{ 
\cbr{ \mu(i),\mu(j)} &= \cbr{ \pistar(a),\pistar(b)} \label{eq: Condition1}\\
\cbr{ \mu(a),\mu(b)} &= \cbr{ \pistar(i),\pistar(j)}. \label{eq: Condition2}
}
Consider then the dependency graph $H$ on the node set $V(H) := \cbr{ \cbr{i,j} \colon (i,j) \in \calN(\mu)}$ such that two nodes in this graph $\cbr{i,j}$ and $\cbr{a,b}$ have an edge between them if and only if they satisfy~\eqref{eq: Condition1} or~\eqref{eq: Condition2}. Since each node in $H$ has at most $2$ neighbors, it follows that $H$ is $3$-colorable. Letting $\calN_1$, $\calN_2$, and $\calN_3$ be the partition corresponding to the $3$ colors, it can be seen that $X_{\calN_1}$, $X_{\calN_2}$ and $X_{\calN_3}$ are independent sums of Binomial random variables. Specifically, for each $(i,j) \in \calN(\mu)$, we have that $\tG_1(i,j) \tG_2(\mu(i),\mu(j)) \sim \mathsf{Bern}(p^2 s^2)$. It follows that for each $m \in \{1,2,3\}$:
\alns{ 
X_{\calN_m} \preceq \text{Bin}(|\calN_{m}(\mu)|,p^2s^2) \preceq \text{Bin}(dn,p^2s^2),
}
where we have used the fact that $|\calN_{m}(\mu)| \leq |\calN(\mu)| \leq |\A(\mu)| \leq dn$. Finally, 
\alns{ 
\P{f(\mu) \geq kd} &\leq \P{2X_{\calT} +X_{\calN} \geq kd} \\
& \leq \P{2\pbr{X_{\calT} + X_{\calN_{1}}+X_{\calN_{2}}+X_{\calN_{3}}} \geq kd} \\
& \leq \sum_{m=1}^3 \P{2X_{\calT} + 6X_{\calN_{m}} \geq kd}.
}
The Chernoff bound can then be applied to the binomial distribution. It follows that for any $\theta > 0$:
\alns{ 
\P{f(\mu) \geq kd} &\leq \sum_{m=1}^3 e^{-\theta kd} \mathbb{E}[e^{2\theta X_{\calT}}] \mathbb{E}[e^{6\theta X_{\calN_{m}}}] \\
&\leq 3e^{-\theta kd}\pbr{1+ps^2(e^{2\theta}-1)}^d \pbr{1+p^2s^2e^{6\theta}-1}^{dn}\\
& \stackrel{(a)}{\leq } 3 \exp\pbr{-d\pbr{\theta k - e^{2\theta}ps^2 - ne^{6\theta}p^2s^2}},
}
where (a) follows from the fact that $(1+x)^t \leq e^{tx}$. Finally, set $\theta = c' \log n$, and observe that when $c < 1/6$:
\alns{ 
e^{2\theta} ps^2 &= n^{2c'-1}Cs^2\log n = o(1), \\
ne^{6\theta}p^2s^2 &= n^{6c-1}C^2s^2\pbr{\log n}^2 = o(1).
}
Finally, since $k \geq 13$, it can be ensured that $ck > 2$ by choosing $c = \frac{1}{6}-\eps$ for sufficiently small $\eps$. This yields
\alns{ 
\xi \leq \P{f(\mu) \geq kd}^{1/d} \leq \pbr{3\exp\pbr{-d(ck \log n - o(1))}}^{1/d} = o(n^{-2}),
}
as desired.
\end{proof}

\subsection{Supporting Lemmas for~\texorpdfstring{\Cref{thm: Achievability-WCER}}{Theorem 8}} \label{subsec: def-proofs2}

\begin{lemma} \label{lem: k-core-degrees}
    Let $n$ and $k$ be positive integers, and let $p,s,\gamma,\spl$ be real numbers such that $0 \leq p,s,\gamma,\spl\leq 1$. Let $(\B_1,\B_2,\tG_1,\tG_2,\pistar)$ be distributed according to $\WCER(n,p,s,\gamma,\spl)$. Suppose that $p = \frac{C\log(n)}{n}$ for some positive constant $C$. Let $\alpha^* = 1- \gamma + \spl(1-\spl)\gamma^2$. If $C > \frac{1}{s^2\alpha^*}$, then the $k$-core $M^*$ of $\tG_1\wedge_{\pistar}\tG_2$ with $k = \sqrt{\log n}$ satisfies
    \alns{ 
    \P{M^* = (\B_1\cup\B_2')^c} = 1- o(1).
    }
\end{lemma}
\begin{proof}
Let the matching $\tpi^*$ with $\dom(\tpi^*) = (\B_1\cup\B_2')^c,$ denote the restriction of the true permutation $\pistar$ to the node set $(\B_1\cup\B_2')^c$. Let $\calH_1$ and $\calH_2$ denote the events 
\alns{ 
\mathcal{H}_1 &\colon  \bigcap_{i \in \B_1\cup\B_2'}\cbr{ \deg_{\tG_1\wedge_{\pistar}\tG_2}(i) < k } \\
\mathcal{H}_2 & \colon \bigcap_{j \notin \B_1\cup\B_2'} \cbr{\deg_{\tG_1\wedge_{\tpi^*}\tG_2}(j) > k }.
}
Note that the intersection graph in $\mathcal{H}_1$ is with respect to $\pistar$ whereas the intersection graph in $\mathcal{H}_2$ is with respect to $\tpi^*$. Let $\mathcal{H} = \mathcal{H}_1\cap\mathcal{H}_2$. It suffices to show that $\P{\mathcal{H}^c} = o(1)$. This implies that with high probability, no node in $\B_1\cup\B_2'$ has degree greater than $k$ in the intersection graph $\tG_1 \wedge_{\pistar}\tG_2$, and therefore cannot belong to its $k$-core. Furthermore, it also implies that with high probability, the subgraph of $\tG_1 \wedge_{\pistar}\tG_2$ induced on the node set $(\B_1\cup\B_2')^c$ has minimum degree at least $k$. In turn, this implies that $\P{M^* = (\B_1\cup\B_2')^c} = 1-o(1)$, as desired.

\paragraph{Bounding $\P{\mathcal{H}_1^c}$} Since the sets $\B_1$ and $\B_2'$ are selected independent of $G_1$ and $G_2$, it follows for any $i \in \B_1\cup\B_2'$ that $\deg_{\tG_1\wedge_{\pistar}\tG_2}(i) \sim \mathsf{Bin}(n-1,p^2s^2)$. This is because for any $i \in \B_1\cup\B_2'$ and $j \in [n]$ such that $j \neq i$, Therefore, $\tG_1(i,j)$ and $\tG_2(\pistar(i),\pistar(j))$ are independent Bernoulli random variables with mean $ps$. Therefore, for any $\delta > 0$:
\aln{ 
\P{\calH_1^c} & =\P{\bigcup_{i\in\B_1\cup\B_2'} \!\!\!\cbr{\deg_{\tG_1 \wedge_{\pistar} \tG_2}(i) > (1\!+\!\delta)np^2s^2)}}  \nonumber \\
& \leq \P{|\B_1\cup\B_2'| > 1.01 (1\!-\!\alpha^*)n} \!+\! \P{ \bigcup_{i\in\B_1\cup\B_2'} \!\!\!\!\cbr{\deg_{\tG_1 \wedge_{\pistar} \tG_2}(i) > (1\!+\!\delta)np^2s^2} \cap \cbr{|\B_1\cup\B_2'| \leq 1.01 (1\!-\!\alpha^*)n}\!} \nonumber \\
&\leq \P{|\B_1\cup\B_2'| > 1.01 (1-\alpha^*)n} + 1.01 (1-\alpha^*)n \times \P{\mathsf{Bin}(n-1,p^2s^2) > (1+\delta)np^2s^2} \nonumber \\ 
& \stackrel{\text{(a)}}{\leq} o(1) + 1.01 (1-\alpha^*) n \times \P{\text{Bin}(n-1,p^2s^2) > (1+\delta)np^2s^2)} \nonumber\\
& \stackrel{\text{(b)}}{\leq} o(1) + 2n \times \P{\text{Bin}(n,p^2s^2) > (1+\delta)np^2s^2)} \nonumber\\
& \stackrel{\text{(c)}}{\leq} o(1) + 2n \times \pbr{\frac{\exp\pbr{\delta}}{\pbr{1+\delta}^{1+\delta}}}^{np^2s^2}, \label{eq: RHS-of-bad-node-degrees}
}
where (a) follows from~\Cref{lem: hypgeom} and (b) follows from $\alpha^* \geq 0$ and the stochastic domination $\mathsf{Bin}(n-1,q) \preceq \mathsf{Bin}(n,q)$ for any $q \in [0,1]$. Finally, (c) follows from~\Cref{lem: concentration-of-binomial}. Setting $\delta = \frac{n}{\pbr{\log n}^2}$ and recalling that $p = \frac{C\log n}{n}$, it follows that $(1+\delta)np^2s^2 = C^2 s^2 + o(1)$. Therefore,~\cref{eq: RHS-of-bad-node-degrees} can be written as
\alns{ 
\P{\calH_1^c} & \leq \P{\bigcup_{i\in \B_1\cup\B_2'} \cbr{\deg_{\tG_1\wedge_{\pistar}\tG_2}(i) > C^2s^2 + o(1)}} \\
&\leq o(1) + 2 n \times \pbr{\frac{\exp\pbr{\frac{n}{(\log n)^2}}}{\pbr{1+\frac{n}{(\log n)^2}}^{1+\frac{n}{(\log n)^2}}}}^{C^2s^2 \frac{(\log n)^2}{n}} \\
&= \begin{cases} 
    o(1),      & Cs > 1 \\
    \omega(1), & Cs < 1
\end{cases}
.
}
Since $C > 1/(s^2\alpha^*)$, it follows that $Cs > \frac{1}{s\alpha^*} \geq 1$. This is because $s\alpha^* \leq 1$. Therefore, for any $\eps > 0$ and any $k'$ such that $k' > C^2 s^2 + \eps$, it is true that
\aln{
\P{\calH_1^c} \leq \P{\bigcup_{i\in\B_1\cup\B_2'} \cbr{\deg_{\tG_1 \wedge_{\pistar} \tG_2}(i) > k'}} = o(1). \label{eq: RHS-A}
}
Indeed, our choice of $k = \sqrt{\log n}$ satisfies $k > C^2s^2 +\eps$ for all sufficiently large $n$. Therefore, $\P{\calH_1^c} = o(1)$ as desired.
 
\paragraph{Bounding $\P{\calH_2^c}$} The probability that a node in $\tG_1\wedge_{\tpi^*}\tG_2$ has degree lesser than $k$ is computed. For any $i \in (\B_1\cup\B_2')^c$, 
\alns{ 
\deg_{\tG_1\wedge_{\tpi^*}\tG_2}(i) = \sum_{\substack{j \in (\B_1\cup\B_2')^c \\ j \neq i}} \tG_1(i,j) \tG_2\pbr{\pistar(i),\pistar(j)}.
}
Since $i$ and $j$ are both in $(\B_1\cup\B_2')^c$, it follows that $\tG_1(i,j) \tG_2\pbr{\pistar(i),\pistar(j)} = G_1(i,j) G_2\pbr{\pistar(i),\pistar(j)} \sim \mathsf{Bern}(ps^2)$. Further, for any $j_1, j_2 \in (\B_1\cup\B_2')^c$ such that $j_1 \neq j_2 \neq i$, it follows by independence across edges that $G_1\pbr{i,j_1}G_2\pbr{\pistar(i),\pistar(j_1)}$ and $G_1\pbr{i,j_2}G_2\pbr{\pistar(i),\pistar(j_2)}$ are independent random variables. Therefore, it follows that $\deg_{G_1\wedge_{\tpi^*}G_2}(i) \sim \mathsf{Bin}(\pbr{1-\gamma}n-1,ps^2)$. We have from a union bound that for any $\delta' \in (0,1)$:
\aln{ 
\P{\calH_2^c} &= \P{\bigcup_{j \in (\B_1\cup\B_2')^c} \cbr{ \deg_{\tG_1 \wedge_{\tpi^*} \tG_2}(i) < \pbr{1-\delta'}\pbr{ \alpha^*n - 1}ps^2 }} \nonumber
\\ & \leq \P{\left| (\B_1\cup\B_2')\right| > 1.01 (1-\alpha^*)n} \nonumber \\
& \qquad + \P{\bigcup_{j \in (\B_1\cup\B_2')^c} \cbr{ \deg_{\tG_1 \wedge_{\tpi^*} \tG_2}(i) < \pbr{1-\delta'}\pbr{ \alpha^*n - 1}ps^2 } \cap \left| (\B_1\cup\B_2')\right| \leq 1.01 (1-\alpha^*)n } \nonumber \\ 
& \stackrel{\text{(d)}}{\leq} o(1) + 2n  \times \P{\mathsf{Bin}\pbr{\alpha^*n - 1,ps^2} < \pbr{1-\delta'}\pbr{ \alpha^*n-1}ps^2 } \nonumber \\
& \stackrel{\text{(e)}}{\leq} o(1) + 2n \times \pbr{\frac{e^{-\delta'}}{\pbr{1-\delta'}^{1-\delta'}}}^{ps^2\pbr{\alpha^*n-1}} \nonumber \\
& \stackrel{\text{(f)}}{=} o(1) + 2n \times \pbr{\frac{e^{-\delta'}}{\pbr{1-\delta'}^{1-\delta'}}}^{\pbr{Cs^2\alpha^* - \frac{Cs^2}{n}}\log n } \nonumber \\ 
& \stackrel{\text{(g)}}{=} 
\begin{cases}
    o(1), & Cs^2\alpha^* > 1 \\
    \omega(1), & \text{otherwise}
\end{cases}
. \label{eq: RHS-good-nodes}
}
Here, in (d),~\Cref{lem: hypgeom} is used along with the fact that $\alpha^* \geq 0$. Further, in (e),~\Cref{lem: concentration-of-binomial} is used, and in (f) we have set $p = C \log(n)/n$. 

To see why (g) is true, consider the function $$g(a,x,n) = n x^{a \log n}.$$
Notice that if $a \log(x) < -1$, then $\lim_{n\to\infty}g(a,x,n) = 0$. Equivalently, this sufficient condition requires $x < \exp\pbr{-1/a}$. Setting $a = Cs^{2}\alpha^* - \frac{Cs^2}{n}$, this condition reduces to 
\aln{
x < \exp\pbr{-\frac{1}{Cs^2\alpha^* - \frac{Cs^2}{n}}}. \label{eq: suff-cond1}
}
When $Cs^2\alpha^* > 1$, there is a sufficiently small constant $\eps > 0$ such that for all $n$ sufficiently large, it is true that $Cs^2\alpha^* - \frac{Cs^2}{n} > 1 + \eps$. Therefore,~\eqref{eq: suff-cond1} holds if $x < \exp\pbr{-\frac{1}{1+\eps}}$. For our purpose, $x$ is a function of $\delta'$:
\alns{ 
x(\delta') = \frac{e^{-\delta'}}{(1-\delta')^{1-\delta'}}.
}
Clearly, $x(0) = 1$ and the right-hand side above is continuous and decreasing on $(0,1)$. Therefore, there exists $\delta'' > 0$ such that for all sufficiently small $\eps$:
$$x(\delta'') = \exp\pbr{-\frac{1}{1+\eps/2}} < \exp\pbr{-\frac{1}{1+\eps}},$$
which satisfies~\eqref{eq: suff-cond1} and hence implies (g). Selecting $\delta' = \delta''$, it follows from~\eqref{eq: RHS-good-nodes} that for any $k'$ such that $k' < (1-\delta')Cs^2\alpha^*\log(n) - o(1)$:
\aln{ 
\P{\calH_2^c} \leq \P{\bigcup_{j\in (\B_1\cup\B_2')^c} \cbr{\deg_{\tG_1 \wedge_{\tpi^*} \tG_2} < k'}} = o(1). \label{eq: RHS-B}
}
Indeed, our choice of $k = \sqrt{\log n}$ satisfies this condition for sufficiently large $n$. Combining~\cref{eq: RHS-A}~and~\cref{eq: RHS-B} via a union bound, it follows that $\P{\calH^c} = o(1)$ as desired. This concludes the proof.
\end{proof}

\begin{lemma} \label{lem: almost-k-core-degrees}
Let $n$ and $k$ be positive integers, and let $p,s,\gamma,\spl$ be real numbers such that $0 \leq p,s,\gamma,\spl\leq 1$. Let $(\B_1,\B_2,\tG_1,\tG_2,\pistar)$ be distributed according to $\WCER(n,p,s,\gamma,\spl)$, and let $\B_2'$ denote the pre-image of $\B_2$ under $\pistar$. Suppose that $p = \frac{C\log(n)}{n}$ for some positive constant $C$. Let $\alpha^* = 1- \gamma + \spl(1-\spl)\gamma^2$. If $C < \frac{1}{s^2\alpha^*}$, then the $k$-core $M^*$ of $\tG_1\wedge_{\pistar}\tG_2$ with $k = \sqrt{\log n}$ satisfies
\alns{
\P{ \left|(\B_1\cup\B_2')^c \setminus M^* \right| = o(n)} = 1 - o(1).
}
\end{lemma}
\begin{proof}
From~\Cref{thm: Achievability-WCER} (i) and~\Cref{thm: WCER-Impossibility} (i), it follows that $\P{M^* \subseteq (\B_1\cup\B_2')^c} = 1 - o(1)$. It remains to show that all but a vanishing fraction of the nodes in $(\B_1\cup\B_2')^c$ are also in $M^*$. Let $\tpi^*$ with $\dom(\tpi^*) = (\B_1\cup\B_2')^c$ denote the restriction of the permutation $\pi^*$ to the node set $(\B_1\cup\B_2')^c$. Let $\core_k(G)$ denote the $k$-core of a graph $G$. First, it is shown that none of the nodes in $\B_1\cup\B_2'$ belong to the $k$-core with high probability. Since for any $i \in \B_1\cup\B_2'$, it holds that $\deg_{\tG_1\wedge_{\pistar}\tG_2}(i) \sim \mathsf{Bin}(n-1,p^2s^2) \preceq \mathsf{Bin}(n,p^2s^2)$, and since $|\B_1\cup\B_2'| \leq |\B_1| + |\B_2'| \leq \gamma n$, it follows by the union bound that for any $\delta > 0$
\alns{ 
\P{ \bigcup_{i \in \B_1\cup\B_2'} \cbr{\deg_{\tG_1\wedge_{\pistar}\tG_2}(i) \geq (1+\delta)n p^2 s^2 }} &\leq \gamma n \times \P{\mathsf{Bin}(n,p^2s^2) > (1+\delta)np^2s^2} \\
& \leq \gamma n \times \pbr{\frac{\exp\pbr{\delta}}{(1+\delta)^{1+\delta}}}^{(n-1)p^2s^2}.
}
Setting $p = C \log(n)/n$ yields
\alns{ 
\P{ \bigcup_{i \in \B_1\cup\B_2'} \cbr{\deg_{\tG_1\wedge_{\pistar}\tG_2}(i) \geq (1+\delta)C^2s^2\frac{\log(n)^2}{n}} } \leq \gamma n\times \pbr{\frac{e^{\delta}}{(1+\delta)^{1+\delta}}}^{C^2s^2 \frac{\log(n)^2}{n}}.
}
Setting $\delta = \frac{n}{C^2s^2\log(n)^{3/2}} - 1$ yields
\aln{ 
\P{ \bigcup_{i \in \B_1\cup\B_2'} \cbr{\deg_{\tG_1\wedge_{\pistar}\tG_2}(i) \geq \sqrt{\log(n)}}} \leq \gamma n \pbr{ \frac{\exp\pbr{\frac{n}{C^2s^2\log(n)^{3/2}}-1}}{\pbr{ \frac{n}{C^2s^2 \log(n)^{3/2}}}^{\frac{n}{C^2s^2\log(n)^{3/2}}}}}^{C^2s^2 \frac{\log(n)^2}{n}} = o(1), \label{eq: o1}
}
whenever $Cs > 0$. Therefore, setting $k = \sqrt{\log(n)}$ yields $\P{\mathcal{E}} = o(1)$, where $\mathcal{E}$ denotes the event that there exists a node in $\B_1\cup\B_2'$ with degree larger than $k$. Therefore, on the event $\mathcal{E}^c$, it follows that no node in $\B_1\cup\B_2'$ belongs to the $k$-core. It remains to show that all but a vanishing fraction of nodes in $(\B_1\cup\B_2')^c$ form the $k$-core. Consider then the trimmed graph $H := \tG_1 \wedge_{\tpi^*} \tG_2$ on the node set $(\B_1\cup\B_2')^c$. Thus, on the event $\mathcal{E}^c$, it is true that 
\aln{ 
\core_k(H) \subseteq M^* \subseteq (\B_1\cup\B_2')^c. \label{eq: Ec}
}
On the other hand, since $\B_1$ and $\B_2'$ are selected independent of the graphs $G_1$ and $G_2$, it follows that the graph $H$ itself is an Erd{\H{o}}s-R{\'e}nyi graph on $|(\B_1\cup\B_2')^c|$ nodes, where each edge is present with probability $Cs\frac{\log n}{n}$. The $k$-core of an Erd{\H{o}}s-R{\'e}nyi graph is a well studied problem. We invoke Theorem~2 of~\cite{luczak1991size}, restated below for convenience.
\begin{proposition} \label{prop: Luczak}
Let $c(n) = (n-1)p$ denote the average degree in an Erd{\H{o}}s-R{\'e}nyi graph $G$ on $n$ nodes with edge probability $p$. For every $\eps > 0$, there is a constant $d$, such that for $c = c(n) > d$ and $k = k(n) \leq c - c^{0.5+\eps}$, it is true that
\alns{ 
\P{|\core_k(G)| \geq n \pbr{1 - \exp\pbr{-c^{\eps}}}} = 1 - o(1).
}
\end{proposition}
First, some intuition is presented. Observe that the average degree $c$ in $H$ is given by $(|(\B_1\cup\B_2')^c|-1)\times \frac{Cs\log(n)}{n} = \Theta(\log(n))$ with high probability. Further, $k = \sqrt{\log(n)} \leq c - c^{0.5+\eps}$ whenever $\eps < 0.5$. Therefore, the graph $H$ with $k  =\sqrt{\log n}$ satisfies the conditions of~\Cref{prop: Luczak} with high probability, and so the $k$-core contains almost all the nodes of $H$ with high probability. Let $\alpha^* = 1 -\gamma + \spl(1-\spl)\gamma^2$. Formally, for any $\delta > 0$:
\aln{ 
\P{ |(\B_1\cup\B_2')^c\setminus M^*| > \delta n} &\leq \P{ \mathcal{E}} + \P{|(\B_1\cup\B_2')^c| < (1-\delta)(1-\alpha^*)n } \nonumber \\
& \qquad \ + \P{ |\core_{\sqrt{\log n}}(H)| < (1-\delta)(1-\alpha^*)n \pbr{1 - \exp\pbr{-c^{0.1}}}} \label{eq: o2}\\
& \leq o(1) + o(1) + o(1),
}
where we have showed in~\eqref{eq: o1} that the first term is $o(1)$. The second term is $o(1)$ due to~\Cref{lem: hypgeom}, and the third term is $o(1)$ due to~\Cref{prop: Luczak}. It is emphasized that~\eqref{eq: o2} holds because $\core_k(H) \subseteq M^* \subseteq (\B_1\cup\B_2')^c$ under the event $\mathcal{E}^c$. This concludes the proof.
\end{proof}

\section{Proofs for the \texorpdfstring{$\SCER$}{SCG} Model} \label{sec-Apx-SCER}
This section analyzes the maximum overlap estimator in the context of the $\SCER$ model. First, a simple concentration inequality is shown for the number of edges in $G_1\wedge_{\pistar}G_2$. Then, the moment generating function of the random variable $X(\pi)$ is bounded using techniques introduced in~\cite{cullina2017exact} and refined in~\cite{wu2022settling}, among others.

\begin{lemma}\label{lem: concentration-of-Xid}
    Let $n$ be a positive integer and $s$ be a real number such that $s \in (0,1]$. Let $C > 0$ be constant and let $p \geq C \log(n)/n$. Let $(G_1,G_2,\pistar)\sim\CER(n,p,s)$. Let $X(\id)$ denote the number of edges in the intersection graph $G_1\wedge_{\pistar}G_2$. Then, for any $\eps \in (0,1)$:
    \alns{ 
    \P{X(\id) \leq (1-\eps)\binom{n}{2}ps^2} = o(1).
     }
\end{lemma}
\begin{proof}
First, notice that $G_1\wedge_{\pistar}G_2$ is also an Erd{\H{o}}s-R{\'e}nyi graph on $n$ nodes, where each edge is present with probability $ps^2$. It follows that the number of edges in the graph $X(\id)\sim\mathsf{Bin}\pbr{\binom{n}{2},ps^2}$. Thus,
\alns{ 
\P{X(\id) \leq \pbr{1-\eps}\binom{n}{2}ps^2 } &\stackrel{\text{(a)}}{\leq} \pbr{\frac{e^{-\eps}}{(1-\eps)^{1-\eps}}}^{\binom{n}{2}ps^2}\\
& \stackrel{\text{(b)}}{\leq} \pbr{\frac{e^{-\eps}}{(1-\eps)^{1-\eps}}}^{\frac{Cs^2}{2}(n-1)\log(n)} \\
& = o(1),
}
where (a) follows from~\Cref{lem: concentration-of-binomial} and (b) is true because $e^{-\eps} < (1-\eps)^{1-\eps}$ for all $\eps \in (0,1)$.
\end{proof}

\begin{lemma}\label{lem: concentration-of-degree} 
Let $s\in (0,1]$ and $\eps > 0$ be fixed, and $C > 0$ be constant. Let $(G_1,G_2) \sim \CER(n,p,s)$ be a correlated pair of Erd{\H{o}}s-R{\'e}nyi graphs, where $p = C \log(n)/n$ and let $\Delta_2$ denote the maximum node degree in $G_2$. If $C > \frac{3}{s \eps^2}$, then 
\alns{ 
\P{ \Delta_2 > (1+\eps)  nps} = o(1).
}
\end{lemma}
\begin{proof}
Let $v_i \in [n]$ denote a node and let $\deg(v_i)$ denote its degree in $G_2$. By the union bound and the fact that $p =  C \log(n)/n$,
\alns{
\P{ \Delta_2 > (1+\eps) nps} &\leq \sum_{i=1}^{n} \P{\deg(v_i) > (1+\eps)\cdot  C s \log(n)} \\
&= n \times \P{ \mathsf{Bin}(n-1,C \log(n)/n) > (1+\eps)\cdot  C s \log(n)}\\
& \leq n \times \P{ \mathsf{Bin}(n,C \log(n)/n) > (1+\eps) \cdot C s \log(n)} \\ 
& \stackrel{\text{(a)}}{\leq} n \times \exp\pbr{- \frac{\eps^2 C s \log(n)}{3}} \\
& = n^{1 - \frac{\eps^2 C s}{3}} \\
& \stackrel{\text{(b)}}{=} o(1),
}
where (a) uses Theorem 4.4 of~\cite{mitzenmacher2017}, and (b) is because $C > \frac{3}{s \eps^2}$.
\end{proof}

\begin{lemma}\label{lem: MGF}
Let $\pi$ be a permutation on $[n]$, and let $(G_1,G_2,\pistar) \sim \CER(n,p,s)$. Let $X(\pi)$ denote the number of edges in $G_1\wedge_{\pi}G_2$. The moment generating function of $X(\pi)$ is given by
\alns{ 
\E{e^{t X(\pi)}} = \prod_{k=1}^{\binom{n}{2}} L_k^{N_k}
}
where $N_k$ is the number of $k$-orbits in the edge decomposition of $\pi$ and
\aln{ 
L_k := \text{Tr}(L^k), \label{eq: Lk-formula}
}
where the $2\times 2$ matrix $L$ is given as
\alns{ 
\begin{bmatrix} 
1-ps & \frac{ps}{1-ps}\pbr{1-p + p(1-s)^2 + ps(1-s)e^t }\\
1-ps & ps\pbr{1-s+se^t}
\end{bmatrix}
}
Moreover, $L_k \leq L_2^{k/2}$.
\end{lemma}

\begin{proof} 
For a permutation $\pi$, Let $\O^{\pi} = \bigcup_{k=1}^{\binom{n}{2}} \O^{\pi}_k$ denote the edge orbit decomposition of $\pi$, where $\O^{\pi}_k$ is the set of all $k$ length edge orbits of $\pi$. By independence across edge orbits:
\alns{ 
\E{\exp\pbr{t X(\pi)}} &= \E{\exp\pbr{t \!\!\! \sum_{\{i,j\}\in \binom{[n]}{2}}\!\! G_1\cbr{i,j} G_2\cbr{\pi(i),\pi(j)}}} \\
&= \E{\exp\pbr{t \!\sum_{O \in \O^{\pi}} \sum_{(i,j) \in O} G_1\cbr{i,j}G_2\cbr{\pi(i),\pi(j)}}}  \\
&= \prod_{O\in\O^{\pi}} \E{\exp\pbr{ t X_O(\pi)}},
}
where $X_O(\pi) = \sum_{i,j \in O} G_1\cbr{i,j}G_2\cbr{\pi(i),\pi(j)}$. Let $(a_i,b_i)_{i=1}^{k+1}$ with $(a_{k+1},b_{k+1}) = (a_1,b_1)$ denote mutually independent random variables so that $a_i, b_i \sim \mathsf{Bern}(ps)$ and $\P{b_i = 1 \mid a_i = 1} = s$. Let $L_k := \E{e^{tX_O}}$ for any edge orbit $O$ such that $|O| = k$. Then,
\aln{ 
L_k &= \E{\exp\pbr{t\sum_{j=1}^{k}a_j b_{j+1} }} = \E{\prod_{j=1}^k \exp\pbr{t \cdot a_j  b_{j+1}}} \nonumber \\
&= \E{ \E{ \prod_{j=1}^k \exp\pbr{t \cdot a_j b_{j+1}} \mid b_1,\cdots,b_k}} \nonumber\\
&= \E{ \E{ \prod_{j=1}^k \exp\pbr{t \cdot a_j b_{j+1}} \mid b_j,b_{j+1}}} \nonumber\\
&\stackrel{(a)}{=} \E{ \prod_{j=1}^{k} \E{\exp\pbr{t \cdot a_j b_{j+1}} \mid b_j,b_{j+1}}} \nonumber\\
& = \sum_{\substack{(\theta_1,\cdots,\theta_k) \in \{0,1\}^k \\ \theta_{k+1}=\theta_1} } \P{(b_1,\cdots,b_k) = (\theta_1,\cdots,\theta_k) } \prod_{j=1}^{k} \E{ \exp\pbr{ t \cdot a_{j}b_{j+1} \mid b_j = \theta_j, b_{j+1} = \theta_{j+1} }} \label{eq: Mat}
}
where (a) follows from conditional independence of edges in $G_1$ given $G_2$. On the other hand,~\cref{eq: Mat} is exactly equal to
\alns{ 
\eqref{eq: Mat} = \text{Tr}(L^k),
}
where the $2\times 2$ matrix $L$ is given by
\aln{ 
L(\ell,m) = \E{\exp\pbr{t\cdot a_1 b_2} \mid b_1 = \ell, b_2 = m} \times \P{b_2 = m}, \quad\ell,m \in \{0,1\} \label{eq: L}
}
Computing the conditional expectations gives as desired,
\alns{ 
\begin{bmatrix} 
1-ps & \frac{ps}{1-ps}\pbr{1-p + p(1-s)^2 + ps(1-s)e^t }\\
1-ps & ps\pbr{1-s+se^t}
\end{bmatrix}.
}
Finally, we show that $L_k \leq L_2^{k/2}$. Since the eigenvalues of $L$ are given by $\frac{T \pm \sqrt{T^2 - 4D}}{2}$, where $T$ and $D$ are the trace and determinant of $L$ respectively, it follows that
\alns{ 
L_k = \text{Tr}(L^k) = \pbr{ \frac{T + \sqrt{T^2 - 4D}}{2}}^k + \pbr{ \frac{T - \sqrt{T^2- 4D}}{2}}^k, 
}
and it follows that $L_k \leq L_2^{k/2}$.
\end{proof}

\begin{corollary} \label{cor: L1L2}
Evaluating the trace of $L$ and $L^2$,
\alns{ 
L_1 &= 1 - ps + ps\pbr{1-s+se^t} \\
L_2 &= (1-ps)^2 + 2ps\pbr{1-p + p(1-s)^2 + ps(1-s)e^t} + (ps\pbr{1-s+se^t})^2.
}
\end{corollary}
\begin{proof}
    This follows from a direct computation using~\cref{eq: Lk-formula}.
\end{proof}

\begin{lemma} \label{lem: Xpi-decays}
Let $n$ be a positive integer and $s,\gamma,\alpha$ be real numbers such that $s \in (0,1]$, and $\gamma,\alpha \in [0,1]$. Let $\mathsf{A}$ be any adversary and let $(\B_1,\B_2,\tG_1,\tG_2,\pistar)$ be distributed according to $\SCER(n,p,s,\gamma,1,\mathsf{A})$. Let $p_3$ be the event defined in~\eqref{eq: p3-X}. If $\gamma < s(1-\alpha^2)/4$, then $p_3 = o(1)$.
\end{lemma}

\begin{proof}
By definition of $p_3$ and the union bound, 
\alns{ 
p_3 \leq \sum_{k=0}^{\alpha n} \P{ \bigcup_{\pi \in \T^{k/n}} \cbr{ X(\pi) \geq (1-\eps)\binom{n}{2} ps^2 - 2\gamma n\pbr{1+\eps}nps}}.
}
The Chernoff bound then yields for any $t > 0$:
\aln{ 
\P{ X(\pi) \geq (1\!-\!\eps)\binom{n}{2} ps^2 \!-\! \gamma n\pbr{1\!+\!\eps}nps} \leq \exp \pbr{-t \sbr{(1\!-\!\eps)\binom{n}{2} ps^2 \!-\! 2\gamma n\pbr{1\!+\!\eps}nps}} \cdot \E{e^{tX(\pi)}}. \label{eq: Chernoff_UB_logdeg}
}
For each $\pi$, let $n_k^{\pi}$ (resp. $N_k^{\pi}$) denote the number of $k$-orbits in the node (resp. edge) decomposition of $\pi$. The MGF of $X(\pi)$ is handled using~\Cref{lem: MGF}:
\aln{ 
\E{e^{tX(\pi)}} &= \prod_{\ell = 1}^{\binom{n}{2}} L_{\ell}^{N_{\ell}^{\pi}} 
 \leq L_1^{\binom{n_1^{\pi}}{2}+n_2^{\pi}} \cdot \prod_{\ell = 2}^{\binom{n}{2}} L_2^{\ell/2} \nonumber \\
 &= L_1^{\binom{n_1^{\pi}}{2} + n_2^{\pi}} \cdot L_2^{\frac{1}{2}\pbr{\binom{n}{2} -\binom{n_1^{\pi}}{2} - n_2^{\pi}}} \label{eq: MGF_UB_logdeg}
}
Substituting~\cref{eq: MGF_UB_logdeg} in~\cref{eq: Chernoff_UB_logdeg} yields
\alns{ 
\P{X(\pi) \geq (1-\eps)\binom{n}{2} ps^2 - 2\pbr{1+\eps}\gamma n^2 ps} \leq \exp\pbr{\zeta(\pi)},
}
where
\aln{ 
\zeta(\pi) = -t\pbr{(1-\eps)\binom{n}{2} ps^2 - 2\pbr{1+\eps} \gamma n^2ps} + \frac{1}{2}\binom{n}{2}\log(L_2) + \frac{1}{2}\pbr{\binom{n_1^{\pi}}{2} + n_2^{\pi}}\log\pbr{\frac{L_1^2}{L_2}} \label{eq: zeta_logdeg}
}
and 
\alns{ 
L_1 &= 1 - ps + ps\pbr{1-s+se^t} \\
L_2 &= (1-ps)^2 + 2ps\pbr{1-p + p(1-s)^2 + ps(1-s)e^t}+ (ps\pbr{1-s+se^t})^2.
}
It is easy to verify that for all $p,s \in [0,1]$ and $t > 0$, the quantity $\frac{L_1^2}{L_2} \geq 1$, and so an upper bound on $\zeta(\pi)$ can be obtained by using the bound $n_2^{\pi} \leq n - n_{1}^{\pi}$. Let $\beta_{\pi}$ denote the fraction of fixed points in the permutation $\pi$. By definition, $n_1^{\pi} = \beta_{\pi} n$. Using $p = \frac{C \log(n)}{n}$ yields
\alns{ 
L_1 &= 1+ \frac{C\cdot (e^t - 1) s^2 \log(n)}{n}, \\
L_2 &= 1 + \frac{C^2 \pbr{e^t - 1} s^2 \pbr{2+ (e^t-1)s^2} \pbr{\log n}^2}{n^2}.
}
Substituting these in~\cref{eq: zeta_logdeg} yields
\alns{ 
\zeta(\pi) \leq T_1(\pi) + T_2(\pi) + T_3(\pi) + T_4(\pi) + T_5(\pi) + T_6(\pi),
}
where
\alns{ 
T_1(\pi) &=-t \pbr{ \frac{(1-\eps) s^2}{2} - 2(1+\eps)\gamma s} \times \frac{C \log n}{n} \times n^2  = \Theta(n \log n)\\
T_2(\pi) &= \frac{\beta_{\pi}^2 n^2}{2}\log(L_1) =\frac{\beta_{\pi}^2 n^2}{2}\log\pbr{1+ \frac{C\cdot (e^t - 1) s^2 \log(n)}{n}}  \\
& \qquad \qquad \qquad \leq \frac{\beta_{\pi}^2 n^2}{2} \times \frac{C (e^t - 1) s^2 \log n}{n} = O(n \log n) \\
T_3(\pi) &= \frac{(1-\beta_{\pi}^2)n^2}{4} \log(L_2) = \frac{(1-\beta_{\pi}^2)n^2}{4} \log\pbr{1 + \frac{C^2 \pbr{e^t - 1} s^2 \pbr{2+ (e^t-1)s^2} \pbr{\log n}^2}{n^2}} \nonumber \\
& \qquad \qquad \qquad \leq \frac{(1-\beta_{\pi}^2)n^2}{4} \times \frac{C^2 \pbr{e^t - 1} s^2 \pbr{2+ (e^t-1)s^2} \pbr{\log n}^2}{n^2} = O( \pbr{\log n}^2) \\
T_4(\pi) &= \frac{t(1-\eps)ps^2}{2}\times n = \frac{C \cdot t(1-\eps)s^2}{2} \log n = \Theta(\log n) \\
T_5(\pi) &= \frac{(2-3\beta_{\pi})n}{2}\log(L_1) = \frac{(2-3\beta_{\pi})n}{2}\log\!\pbr{1 \!+\! \frac{C (e^t - 1) s^2 \log(n)}{n}} \\
& \qquad \qquad \qquad \leq \frac{(2-3\beta_{\pi})n}{2} \times \frac{C (e^t - 1) s^2 \log(n)}{n} = O(\log n) \\
T_6(\pi) &= -\frac{3(1-\beta_{\pi})n}{4}\log\pbr{L_2} = -\frac{3(1-\beta_{\pi})n}{4}\log\pbr{1 + \frac{C^2 \pbr{e^t - 1} s^2 \pbr{2+ (e^t-1)s^2} \pbr{\log n}^2}{n^2}} \nonumber \\
& \qquad \qquad \qquad \leq -\frac{3(1-\beta_{\pi})n}{4} \times \frac{C^2 \pbr{e^t - 1} s^2 \pbr{2+ (e^t-1)s^2} \pbr{\log n}^2}{n^2} = O\pbr{\frac{(\log n)^2}{n}}.
}
Since the dominant terms are $T_1(\pi)$ and $T_2(\pi)$, it follows that for sufficiently large $n$ and any $t > 0$:
\aln{ 
\sum_{i=1}^{6} T_i(\pi) \leq \pbr{1+\eps}\pbr{ \frac{-t(1-\eps)s^2}{2} + 2t(1+\eps)\gamma s + \frac{\beta_{\pi}^2 s^2}{2}\pbr{e^t - 1}}   \times C n \log n. \label{eq: T1plusT2}
}
Next, the condition in the hypothesis of the theorem is invoked, i.e. $\gamma < s(1-\alpha^2)/4$. Since $\beta_{\pi} \leq \alpha < 1$ for all $\pi \in\T^{\leq\alpha}$, it follows also that $\gamma < s(1-\beta_{\pi}^2)/4$. Let
\aln{ 
t^* = \log\pbr{\frac{1}{\beta_{\pi}^2} \pbr{1 - \frac{4\gamma }{s}}}. \label{eq: t}
}
Note that $t^* > 0$. Also note that
\aln{ 
-\frac{t^*s^2}{2} + 2t^*\gamma s + \frac{\beta_{\pi}^2}{2}s^2 \pbr{e^{t^*} - 1} = -\frac{s}{2}\pbr{s\beta_{\pi}^2 + 4\gamma -s + \pbr{s-4\gamma} \log\pbr{s-4\gamma} - \pbr{s-4\gamma}\log\pbr{s\beta_{\pi}^2}},\label{eq: RHS}
}
and that the RHS of~\cref{eq: RHS} is strictly negative whenever the condition $\gamma < s(1-\beta_{\pi}^2)/4$ is satisfied (see~\Cref{lem: aux}). Therefore, there exists a sufficient small $\eps > 0$ and a sufficiently large $C > 0$ such that
\alns{ 
\sum_{i=1}^{6} T_i(\pi) \leq\pbr{1+\eps} \pbr{ \frac{-t^*(1-\eps)s^2}{2} + 2t^*(1+\eps)\gamma s + \frac{\beta_{\pi}^2 s^2}{2}\pbr{e^{t^*} - 1}}  \times C n \log n < -\frac{\beta_{\pi} + 1}{2} n \log(n).
}
This yields that for sufficiently large $n$:
\alns{ 
\zeta(\pi) \leq -\frac{\beta_{\pi} + 1}{2} n \log n.
}
Finally, this yields
\alns{ 
p_3 &\leq  \sum_{k=0}^{\alpha n} \sum_{\pi \in \T^{k/n}}\P{ X(\pi) \geq (1-\eps)\binom{n}{2} ps^2 - 2\pbr{1+\eps}\gamma n^2 ps} \\
& =  \sum_{k=0}^{\alpha n} n^k \exp \pbr{- \frac{1}{2}\pbr{\frac{k}{n} + 1} n\log(n)} \\
& =  \sum_{k=0}^{\alpha n} \exp \pbr{-\frac{1}{2}(n-k)\log n} \\
& =  \sum_{k=(1-\alpha )n}^{n} \exp \pbr{ - \frac{1}{2} k \log n} \\
& \leq  \frac{\exp\pbr{-\frac{1-\alpha}{2}n\log n}}{1 - \exp\pbr{-\frac{1}{2}\log n}} \\
& = o(1)
}
This concludes the proof.
\end{proof}

\begin{lemma}\label{lem: aux}
    Suppose $\beta_{\pi},\gamma$ are constants such that $\beta_{\pi}> 0$ and $\gamma < s(1-\beta_{\pi}^2)/4$. Then,
    \alns{ 
    s\beta_{\pi}^2 + 4\gamma -s + \pbr{s-4\gamma} \log\pbr{s-4\gamma} - \pbr{s-4\gamma}\log\pbr{s\beta_{\pi}^2} > 0
    }
\end{lemma}
\begin{proof} 
Let $x = s\beta_{\pi}^2$ and $y = s - 4\gamma$. Further, let $z = y/x$. Then,
\alns{ 
s\beta_{\pi}^2 + 4\gamma -s + \pbr{s-4\gamma} \log\pbr{s-4\gamma} - \pbr{s-4\gamma}\log\pbr{s\beta_{\pi}^2} = x(1 - z + z\log(z)).
}
First, note that $x > 0$. Furthermore, $z > 1$ since $\gamma < \frac{s(1-\beta_{\pi}^2)}{4}$ implies $x < y$. Finally, note that the function $z \mapsto 1 - z + z\log z$ is convex and minimized at $z = 1$, where it equals 0. It follows that the desired quantity is positive whenever $z > 1$.
\end{proof}

\begin{lemma} \label{lem: Xpi-decays-lin}
Let $n$ be a positive integer and $p,s,\gamma,\alpha$ be real numbers such that $p\in(0,1)$, $s \in (0,1]$, and $\gamma,\alpha \in [0,1]$. Let $\mathsf{A}$ be any adversary and let $(\B_1,\B_2,\tG_1,\tG_2,\pistar)$ be distributed according to $\SCER(n,p,s,\gamma,1,\mathsf{A})$. Let $\eps \in (0,1)$ and let $p_4$ be the event defined in~\eqref{eq: p4-X}. If 
\alns{ 
\gamma < 1- \sqrt{1-\frac{s^2p(1-p)(1-\alpha^2)}{2}},
}
then $p_4 = o(1)$.
\end{lemma}
\begin{proof}
By definition of $p_4$ and the union bound,
\alns{ 
p_4 &= \P{ \bigcup_{k=0}^{\alpha n} \bigcup_{\pi \in \T^{k/n}} \cbr{X(\pi) \geq (1-\eps)\binom{n}{2}ps^2 - 2\sbr{ \binom{\gamma n}{2} - \gamma (1-\gamma) n ^2}}} \\
& \leq \sum_{k=0}^{\alpha n} \P{\bigcup_{\pi \in \T^{k/n}} \cbr{X(\pi) \geq (1-\eps)\binom{n}{2}ps^2 - 2\sbr{\binom{\gamma n}{2} - \gamma (1-\gamma) n ^2}}}
}
The Chernoff bound is used below to bound these terms. It follows for any $ t > 0$:
\aln{ 
&\P{X(\pi) \geq (1-\eps) \binom{n}{2}ps^2 - 2\sbr{\binom{\gamma n}{2} - \gamma (1-\gamma) n ^2}} \nonumber \\
&\qquad \qquad \qquad \qquad \leq \exp\pbr{-t \pbr{(1-\eps)\binom{n}{2}ps^2 - 2\sbr{ \binom{\gamma n}{2} - \gamma (1-\gamma) n ^2}}} \cdot \E{e^{t X(\pi)}} \label{eq: Chernoff_UB}
}
For each $\pi$, let $n_k^{\pi}$ (resp. $N_k^{\pi}$) denote the number of $k$-orbits in the node (resp. edge) decomposition of $\pi$. The MGF of $X(\pi)$ is handled using~\Cref{lem: MGF}:
\aln{ 
\E{e^{tX(\pi)}} &= \prod_{\ell = 1}^{\binom{n}{2}} L_{\ell}^{N_{\ell}^{\pi}} 
 \leq L_1^{\binom{n_1^{\pi}}{2}+n_2^{\pi}} \cdot \prod_{\ell = 2}^{\binom{n}{2}} L_2^{\ell/2} 
 = L_1^{\binom{n_1^{\pi}}{2} + n_2^{\pi}} \cdot L_2^{\frac{1}{2}\pbr{\binom{n}{2} -\binom{n_1^{\pi}}{2} - n_2^{\pi}}} \label{eq: MGF_UB}
}
Substituting~\cref{eq: MGF_UB} in~\cref{eq: Chernoff_UB} yields
\alns{ 
\P{X(\pi) \geq (1-\eps)\binom{n}{2}ps^2 - 2\sbr{\binom{\gamma n}{2} - \gamma (1-\gamma) n ^2}} \leq \exp\pbr{\zeta(\pi)},
}
where
\alns{ 
\zeta(\pi) = -t\pbr{(1-\eps)\binom{n}{2}ps^2 - 2\sbr{\binom{\gamma n}{2} - \gamma (1-\gamma) n^2}} + \frac{1}{2}\binom{n}{2}\log(L_2) + \frac{1}{2}\pbr{\binom{n_1^{\pi}}{2} + n_2^{\pi}}\log\pbr{\frac{L_1^2}{L_2}}
}
Let $\beta_{\pi}$ denote the fraction of fixed points in the permutation $\pi$. By definition, $n_1^{\pi} = \beta_{\pi} n$. Letting $\tp = p(1-\eps)$ and using the fact that $n_2^\pi \leq n - n_1^{\pi}$, it follows that
\aln{ 
\zeta(\pi) &\leq \pbr{-\frac{t}{2} \pbr{\tp s^2 + 2\gamma^2 - 4\gamma} + \frac{1}{4} \log(L_2) + \frac{\beta_{\pi}^2}{4}\log\pbr{\frac{L_1^2}{L_2}}}  n^2 \label{eq: nsquared} \\
& +  \pbr{ - \frac{t}{2} \pbr{2\gamma - \tp s^2} - \frac{1}{4}\log(L_2) + \frac{2-3\beta_{\pi}}{4}\log\pbr{\frac{L_1^2}{L_2}}}  n 
}
Next, it is shown that if~\eqref{eq: condition-SCER-lin} is satisfied, then there exists $\delta > 0$ such that coefficient of the $n^2$ term in~\eqref{eq: nsquared} is less than or equal to $-\delta$. To that end, notice that this coefficient is strictly negative when
\aln{ 
 \tp s^2 + 2\gamma^2 - 4\gamma > \inf_{t > 0} \cbr{ \frac{1}{t} \pbr{\frac{1}{2} \log(L_2) + \frac{\beta_{\pi}^2}{2}\log\pbr{\frac{L_1^2}{L_2}}} } = \inf_{t> 0}\cbr{\log \pbr{ L_1 ^{\frac{\beta_{\pi}^2}{t}} \cdot L_2^{\frac{1-\beta_{\pi}^2}{2t}} } } \label{eq: Condition}
}
Therefore,~\eqref{eq: Condition} holds if
\aln{ 
\exp\pbr{\tp s^2 + 2\gamma^2 - 4\gamma} &> \exp \pbr{ \inf_{t > 0} \log\pbr{ L_1^{\frac{\beta_{\pi}^2}{t}} \cdot L_2^{\frac{1-\beta_{\pi}^2}{2t}}}} \label{eq: condition_exponent}\\ 
&= \inf_{t>0} L_1^{\frac{\beta_{\pi}^2}{t}} \cdot L_2^{\frac{1-\beta_{\pi}^2}{2t}}  \label{eq: condition_exponent2}
}
since $L_1, L_2 \geq 1$ and therefore the objective function being infimized in~\cref{eq: condition_exponent} is always non-negative. 
Observing that the objective function for the infimization in~\eqref{eq: condition_exponent2} is monotonically increasing in $t$, and substituting for $L_1$ and $L_2$ from~\Cref{cor: L1L2}, it follows that
\alns{
\eqref{eq: condition_exponent2} &= \lim_{t\to 0}\cbr{ \pbr{1 \!-\!ps \!+\! ps(1\!-\!s\!+\!se^t)}^{\frac{\beta_{\pi}^2}{t}} \pbr{(1\!-\!ps)^2 \!+\! 2ps\pbr{1\!-\!p\!+\!p(1\!-\!s)^2 + ps(1\!-\!s)e^t} \!+\! \pbr{ps(1\!-\!s\!+\!se^t)}^2}^{\frac{1-\beta_{\pi}^2}{2t}} }\\
& = \exp\pbr{ps^2 \pbr{p + \beta_{\pi}^2 - p\beta_{\pi}^2}},
}
i.e. the condition~\eqref{eq: Condition} holds if $\exp\pbr{\tp s^2 + 2\gamma^2 - 4\gamma} > \exp\pbr{ps^2 \pbr{p + \beta_{\pi}^2 - p\beta_{\pi}^2}}$. Rewriting this as a condition on $\gamma$ yields that~\cref{eq: Condition} holds whenever 
\aln{ 
\gamma < 1 - \sqrt{1 - \frac{ps^2\pbr{ (1-p)(1-\beta_{\pi}^2) - \eps}}{2} }. \label{eq: condition-SCER-lin-true}
}
Indeed, since $\beta_{\pi} \leq\alpha$ for all $\pi \in \T^{\leq\alpha}$, it follows that whenever~\eqref{eq: condition-SCER-lin} holds, $\eps$ can be chosen sufficiently small so that~\eqref{eq: condition-SCER-lin-true} is satisfied for all $\pi \in \T^{\leq\alpha}$. It follows that there exists $\delta > 0$ such that the coefficient of the $n^2$ term in~\eqref{eq: nsquared} is less than or equal to $-\delta$. Therefore, for sufficiently large $n$:
\alns{ 
\zeta(\pi) &\leq - \frac{\delta}{2} n^2  \leq -\frac{1}{2}\pbr{\beta_{\pi}+ 1}n\log(n).
}
Combining all the above, 
\alns{ 
p_4 &\leq  \sum_{k=0}^{\alpha n} \sum_{\pi \in \T^{k/n}}\P{ X(\pi) \geq \binom{n}{2}\tp s^2 - 2\sbr{\binom{\gamma n}{2} - \gamma(1-\gamma)n^2}} \\
& \leq \sum_{k=0}^{\alpha n} n^k \exp \pbr{- \frac{1}{2}\pbr{\frac{k}{n} + 1} n\log(n)} \\
& = \sum_{k=0}^{\alpha n} \exp \pbr{-\frac{1}{2}(n-k)\log n} \\
& = \sum_{k=(1-\alpha )n}^{n} \exp \pbr{ - \frac{1}{2} k \log n} \\
& \leq \frac{\exp\pbr{-\frac{1-\alpha}{2}n\log n}}{1 - \exp\pbr{-\frac{1}{2}\log n}} \\
& = o(1)
}
as desired. This concludes the proof.
\end{proof}

\section{The \texorpdfstring{$\SCER$}{SCG} Model when both networks are compromised} \label{apx: adv}
In this section, the maximum overlap estimator is analyzed for the case when the adversary corrupts the same number of nodes in both graphs, i.e. $\spl = 1/2$. The action of the adversary is described by Algorithm~\ref{Alg: Overwhelm_Adversary}. A similar adversary can cause the maximum overlap estimator to fail horribly when $\spl \in (0,1)$ but $\spl \neq 1/2$. However, the setting of $\spl = 1/2$ conveys the main ideas without complicating notation and will be the focus of the analysis.

\begin{theorem}
Let $(\B_1,\B_2,\tG_1,\tG_2,\id)$ be distributed according to $\SCER(n,p,s,\gamma,1/2, \mathsf{A})$, where the adversary $\mathsf{A}$ is described in Algorithm~\ref{Alg: Overwhelm_Adversary}. Let $\pihatMO$ denote the matching output by the maximum overlap estimator $\widehat{\mathcal{E}}_{\mathsf{MO}}(\tG_1,\tG_2)$. Then, for any $\eps > 0$, 
\alns{ 
\P{ \ov(\pihatMO, \id) \geq \eps n } = o(1).
}
\end{theorem}
\begin{proof}

Assume without loss of generality that $\gamma \leq 1/2$, since the same argument works by appropriately interchanging the role of $\gamma$ and $1-\gamma$. For any permutation $\pi$, let $X(\pi)$ denote the number of edges in the intersection graph $G_1\wedge_{\pi}G_2$. Similarly, let $\tX(\pi)$ denote the number of edges in the intersection graph $\tG_1\wedge_{\pi}\tG_2$. For $j\in \{1,2\}$, let $E(\tG_j)$ denote the set of edges present in $\tG_j$, and furthermore, let $\widetilde{E}(\tG_j) \subseteq E(\tG_j)$ denote the edges added by the adversary in Algorithm~\ref{Alg: Overwhelm_Adversary}.

Let $\tpi$ be a matching with $\dom(\tpi) = [n]$ such that
\alns{ 
\tpi(i) \in 
\begin{cases} 
\B_2, & i \in \B_1 \\
\B_1, & i \in \B_2 \\
\G_2, & i \in \G_1 \\
\G_1, & i \in \G_2
\end{cases},
}
where $\G_1$, $\G_2$ are as defined in Algorithm~\ref{Alg: Overwhelm_Adversary}. Since $|\B_1| = |\B_2|$ and $|\G_1| = |\G_2|$, such a matching exists. Note that $\ov(\tpi,\id) = 0$. First, a lower bound on $\tX(\tpi)$ is computed.
\alns{ 
\tX(\tpi) &= \left|E\pbr{\tG_1\wedge_{\tpi}\tG_2}\right|  \\
& \geq \left| \widetilde{E}\pbr{\tG_1\wedge_{\tpi} \tG_2}\right| \\
& \stackrel{\text{(a)}}{=} \left| \B_1 \right| \times \left| \G_1 \right| \\
&= \frac{\gamma n}{2} \times \frac{(1-\gamma)n}{2} \\
& = \frac{\gamma(1-\gamma)}{4}\times n^2,
}
where (a) follows from the fact that each node in $\B_1$ connects to every node in $\G_1$ in the graph $\tG_1\wedge_{\tpi}\tG_2$. 

Thus, in order to complete the proof, it suffices to show that
\aln{ 
\P{ \bigcup_{\eps \in (0,1]} \bigcup_{\substack{\pi \\ \ov(\pi,\id) = \eps n} } \tX(\pi) > \frac{\gamma(1-\gamma)}{4}\times n^2} = o(1). \label{eq: STP-adv}
}

A union bound argument is presented below to show~\eqref{eq: STP-adv}. Before proceeding, let us collect some observations about the degrees of nodes in $\tG_1$ and $\tG_2$.
\begin{remark} \label{rem: deg}
Let $\B_1,\B_2,\G_1,\G_2$ be as defined in Algorithm~\ref{Alg: Overwhelm_Adversary}. Since the adversary only adds edges and does not remove any edges,
\alns{ 
\deg_{\tG_1}(i) \leq 
\begin{cases}
 \deg_{G_1}(i) + \frac{(1-\gamma)n}{2}, & i \in \B_1 \\
 \deg_{G_1}(i) + \frac{\gamma n}{2}, & i \in \G_1 \\
 \deg_{G_1}(i), & i \in \B_2 \cup \G_2
\end{cases}
}
and 
\aln{ 
\deg_{\tG_2}(j) \leq 
\begin{cases}
 \deg_{G_2}(j) + \frac{(1-\gamma)n}{2}, & j \in \B_2 \\
 \deg_{G_2}(j) + \frac{\gamma n}{2}, & j \in \G_2 \\
 \deg_{G_2}(j), & j \in \B_1 \cup \G_1
\end{cases}
}
\end{remark}

Continuing, let $([n],\pi')$ be a matching such that $\ov(\pi',\id) = \eps n$. Then,
\aln{ 
\tX(\pi') & = \frac{1}{2} \sum_{i\in[n]} \deg_{\tG_1\wedge_{\pi'}\tG_2} (i) \leq \frac{1}{2} \sum_{i\in [n]} \min\pbr{\deg_{\tG_1}(i), \deg_{\tG_2}(\pi'(i))}  = \frac{1}{2} \pbr{Z_1 + Z_2}, \label{eq: ZX}
}
where
\alns{
Z_1 &:= \sum_{ \substack{ i \in[n] \\ \pi'(i) = i}} \min\pbr{\deg_{\tG_1}(i), \deg_{\tG_2}(i)}, \\
Z_2 &:= \sum_{ \substack{ i \in[n] \\ \pi'(i) \neq i}} \min\pbr{\deg_{\tG_1}(i), \deg_{\tG_2}(\pi'(i))}
}
First, the term $Z_1$ is analyzed. Clearly, for any $i\in[n]$,~\Cref{rem: deg} yields
\alns{ 
\min\pbr{\deg_{\tG_1}(i), \deg_{\tG_2}(i)} \leq  \min\pbr{\deg_{G_1}(i), \deg_{G_2}(i)} ,
}
and so it follows that 
\aln{ 
Z_1 \leq \eps n \times \max_{i \in [n]} \pbr{ \min\pbr{\deg_{G_1}(i), \deg_{G_2}(i)} } \leq  n \times \max_{i\in[n]} \deg_{G_1}(i).\label{eq: Z1-bound}
}
Next, the term $Z_2$ is analyzed. Notice that it can be upper bounded as
\aln{ 
Z_2 \leq Z_{2,1} + Z_{2,2} + Z_{2,3} + Z_{2,4}, \label{eq: Z2-split}
}
where each term is defined and then bounded using~\Cref{rem: deg} as
\aln{ 
Z_{2,1} &:= \sum_{ \substack{i \in [n]  \\ i \in \B_2 \cup \G_2}} \min\pbr{\deg_{\tG_1}(i), \deg_{\tG_2}(\pi'(i))} \leq \left| \B_2\cup\G_2\right| \max_{i\in [n]} \pbr{ \deg_{G_1}(i)} \leq n \max_{i \in[n]} \pbr{\deg_{G_1}(i)},  \\
Z_{2,2} &:= \!\!\!\!\sum_{ \substack{i \in [n] \\ \pi'(i) \in \B_1 \cup \G_1 }} \!\!\!\!\min\pbr{\deg_{\tG_1}(i), \deg_{\tG_2}(\pi'(i))} \leq\left| \B_1\cup\G_1\right| \max_{\pi'(i) \in [n]} \pbr{ \deg_{G_2}(\pi'(i))} \leq n \max_{i \in[n]} \pbr{\deg_{G_2}(i)}, \\
Z_{2,3} &:=\!\!\!\!\sum_{ \substack{i \in [n] \\ i \in \B_1 \\ \pi'(i) \in \B_2 \cup \G_2 }} \!\!\!\!\min\pbr{\deg_{\tG_1}(i), \deg_{\tG_2}(\pi'(i))} \leq \left|\cbr{i: i \in \B_1, \pi'(i)\in\B_2 \!\cup\! \G_2 } \right| \sbr{ \frac{(1\!-\!\gamma )n}{2} + \max_{i\in [n]}\pbr{ \deg_{G_1}(i)}}, \label{eq: Z23}\\
Z_{2,4} &:=\!\!\!\!\sum_{ \substack{i \in [n] \\ i \in \G_1 \\ \pi'(i) \in \B_2 \cup \G_2 }} \!\!\!\!\min\pbr{\deg_{\tG_1}(i), \deg_{\tG_2}(\pi'(i))} \leq \left|\cbr{i: i \in \G_1, \pi'(i)\in\B_2 \!\cup\! \G_2 } \right| \sbr{ \frac{\gamma n}{2} + \max_{i\in [n]}\pbr{ \deg_{G_1}(i)}}. \label{eq: Z24}
}
Let $\eps_1(\pi')$ denote the fraction of nodes $i$ in $\B_1$ such that $\pi'(i) \in \B_2 \cup \G_2$. Similarly, let $\delta_1(\pi')$ denote the fraction of nodes $j$ in $\G_1$ such that $\pi'(j) \in \B_2\cup\G_2$. Therefore,~\eqref{eq: Z23} and~\eqref{eq: Z24} can be written as
\aln{ 
Z_{2,3} &\leq \eps_1(\pi') |\B_1| \pbr{\frac{(1\!-\!\gamma )n}{2} + \max_{i\in [n]}\pbr{ \deg_{G_1}(i)}}  \leq \eps_1(\pi') \frac{\gamma(1-\gamma)}{4}n^2 + n\max_{i\in[n]}\pbr{\deg_{G_1}(i)}, \\
Z_{2,4} &\leq \delta_1(\pi') |\G_1| \pbr{ \ \ \ \frac{\gamma n}{2} \ \ \ \ + \max_{i\in [n]}\pbr{ \deg_{G_1}(i)}}  \leq \delta_1(\pi') \frac{\gamma(1-\gamma)}{4}n^2 + n\max_{i\in[n]}\pbr{\deg_{G_2}(i)}. \label{eq: Z24-bound}
}
Therefore, combining~\eqref{eq: Z2-split}-\eqref{eq: Z24-bound} yields
\aln{ 
Z_2 \leq 2n \max_{i\in[n]}\pbr{\deg_{G_1}(i)} + 2n \max_{i\in[n]}\pbr{\deg_{G_2}(i)} + \pbr{\eps_1(\pi') + \delta_1(\pi')} \frac{\gamma(1-\gamma)}{4} \times n^2 \label{eq: Z2-bound}
}
Substituting~\eqref{eq: Z1-bound} and~\eqref{eq: Z2-bound} in~\eqref{eq: ZX}, 
\aln{ 
\tX(\pi') \leq \frac{1}{2}\pbr{Z_1 + Z_2} \leq \frac{1}{2}\pbr{3n \max_{i\in[n]}\deg_{G_1}(i) + 2n \max_{i\in[n]}\pbr{\deg_{G_2}(i)} + \pbr{\eps_1(\pi') + \delta_1(\pi')} \frac{\gamma(1-\gamma)}{4} \times n^2}. \label{eq: ZZ}
}
The following claim is made next.
\begin{claim} \label{claim: e+d<2}
Let $\eps > 0$ and $\pi'$ be any permutation such that $\ov(\pi',\id) \geq \eps n$. Let $\eps_1$ and $\delta_1$ be as defined above. Then, $\eps_1(\pi') + \delta_1(\pi') < 2$.
\end{claim}
\begin{proof}
    The proof of the claim relies on a simple observation: If node $i$ is not a fixed point of $\pi'$, then neither is the node $\pi'(i)$. Specifically, for each node $i \in \B_1\cup\G_1$ such that $\pi'(i) \in \B_2\cup\G_2$, there exists a unique node $k = \pi'(i)$ in $\B_2\cup\G_2$ such that $\pi'(k)\neq k$. This is true simply because $\pi'$ is a permutation and therefore bijective.

    Assume to the contrary that $\eps_1(\pi') + \delta_1(\pi') = 2$. Since $\eps_1(\pi')$ and $\delta_1(\pi')$ are fractions, our assumption also implies that $\eps(\pi') = 1$ and $\delta_1(\pi') = 1$. Since the number of wrongly matched nodes in $\B_1$ (resp. $\G_1$) is at least as large as the number of nodes in $\B_1$ (resp. $\G_1$) that are mapped to $\B_2\cup\G_2$, it follows that
    \alns{ 
    \left|\cbr{i \in [n]: \pi'(i) \neq i}\right| &\geq 2 \pbr{ \eps_1(\pi')|\B_1| + \delta_1(\pi')|\G_1| } \label{eq: factor-of-2}\\
    &= 2 \pbr{|\B_1| + |\G_1|} \\
    & = n,
    }
    which contradicts the fact that $\ov(\pi', \id) \geq \eps n$. Here, the factor of $2$ in~\eqref{eq: factor-of-2} is due to the aforementioned observation. It is concluded that $\eps_1(\pi') + \delta_1(\pi') < 2$.
\end{proof}
\Cref{claim: e+d<2} confirms the existence of a small positive constant $\eps'$ such that $\eps_1(\pi') + \delta_1(\pi') \leq 2 - \eps'$. Combining this with~\eqref{eq: ZZ} and using the fact that $2n\leq 3n$ for all $n$  yields
\aln{ 
\tX(\pi') \leq \frac{3n}{2} \pbr{\max_{i\in[n]}\deg_{G_1}(i) + \max_{i\in[n]}\deg_{G_2}(i)} + (1- \eps'/2)\times\pbr{ \frac{\gamma(1-\gamma)}{4} \times n^2}.
}
Therefore,
\alns{ 
\P{ \tX(\pi') > \frac{\gamma(1-\gamma)}{4}\times n^2} &\leq \P{ \frac{3n}{2} \pbr{\max_{i\in[n]}\deg_{G_1}(i) + \max_{i\in[n]}\deg_{G_2}(i)} > \frac{\eps'}{2}\times \frac{\gamma(1-\gamma)}{4}\times n^2 } \\
& = \P{ \max_{i\in[n]}\deg_{G_1}(i) + \max_{i\in[n]}\deg_{G_2}(i) > \frac{\gamma(1-\gamma)}{12}\times \eps 'n } \\
& \leq \P{ \max_{i\in[n]}\deg_{G_1}(i) > \frac{\gamma(1-\gamma)}{24} \times \eps' n} + \P{ \max_{i\in[n]}\deg_{G_2}(i) > \frac{\gamma(1-\gamma)}{24} \times \eps' n} \\
& \stackrel{\text{(a)}}{\leq} 2n \times \P{ \mathsf{Bin}(n,ps^2) > \frac{\gamma(1-\gamma)}{24}\times \eps' n}, \\
& \stackrel{\text{(b)}}{\leq} 2n \times \P{ \mathsf{Bin}(n,ps^2) > n}
}
where (a) uses a union bound, the stochastic dominance of $\mathsf{Bin}(n,ps^2)$ over $\mathsf{Bin}(n-1,ps^2)$, and (b) uses the fact that $\eps'\gamma(1-\gamma)/24 < 1$. We use the tail bound on the Binomial distribution~\cite{mitzenmacher2017} 
\alns{ 
\P{\mathsf{Bin}(n,ps^2) > (1+\delta)nps^2} \leq \exp\pbr{-nps^2\delta^2/3},
}
with $\delta = \frac{n}{Cs^2\log(n)}-1$ to get
\alns{ 
\P{ \tX(\pi') > \frac{\gamma(1-\gamma)}{4}\times n^2} \leq 2n \times \P{ \mathsf{Bin}(n,ps^2) > n} \leq 2n \times \exp\pbr{ - \frac{n^2}{3Cs^2\log(n)} }.
}
It remains to perform a union bound over all $\pi'$ such that $\ov(\pi',\id) > \eps n$. However, the number of such permutations is trivially upper bounded by $n! < \exp\pbr{n^{1.1}\log(n)}$ for all sufficiently large $n$. Therefore,
\alns{ 
\P{ \bigcup_{\eps \in (0,1]} \bigcup_{\substack{\pi \\ \ov(\pi,\id) = \eps n} } \tX(\pi) > \frac{\gamma(1-\gamma)}{4}\times n^2} \leq 2n \exp\pbr{n^{1.1}\log(n)} \times \exp\pbr{ - \frac{n^2}{3Cs^2\log(n)} } = o(1), 
}
as desired. This completes the proof.
\end{proof}

\end{document}